\documentclass[11pt,reqno]{amsart}

\usepackage[a4paper,margin=1.08in]{geometry}
\usepackage{amsmath,amssymb,amsthm,mathtools,mathrsfs}
\usepackage{microtype}
\microtypesetup{expansion=false}
\usepackage{enumitem}
\usepackage{array,booktabs}
\setlist[enumerate]{label=\textup{(\roman*)},leftmargin=*,itemsep=.4em}
\usepackage{xcolor}
\usepackage[colorlinks=true,linkcolor=blue!55!black,citecolor=blue!55!black,urlcolor=blue!55!black]{hyperref}
\hypersetup{
 pdftitle={Extreme values of quadratic Hecke L-functions over global fields},
 pdfauthor={Zikang Dong, Long Liu},
 pdfsubject={Quadratic Hecke L-functions, low-energy resonance, and critical windows}
}

\numberwithin{equation}{section}

\newtheorem{theorem}{Theorem}[section]
\newtheorem{proposition}[theorem]{Proposition}
\newtheorem{lemma}[theorem]{Lemma}
\newtheorem{corollary}[theorem]{Corollary}

\theoremstyle{definition}

\theoremstyle{remark}
\newtheorem{remark}[theorem]{Remark}

\DeclareMathOperator{\Res}{Res}
\newcommand{\Norm}{\mathrm N}
\newcommand{\OO}{\mathcal O}
\newcommand{\Om}{\Omega}
\newcommand{\one}{\mathbf 1}
\newcommand{\eps}{\varepsilon}
\newcommand{\p}{\mathfrak p}
\newcommand{\q}{\mathfrak q}
\newcommand{\aideal}{\mathfrak a}
\newcommand{\mideal}{\mathfrak m}
\newcommand{\nideal}{\mathfrak n}
\newcommand{\fideal}{\mathfrak f}
\newcommand{\kideal}{\mathfrak k}
\newcommand{\lideal}{\mathfrak l}
\newcommand{\uideal}{\mathfrak u}
\newcommand{\videal}{\mathfrak v}
\newcommand{\cideal}{\mathfrak c}

\title[Extreme values of quadratic Hecke $L$-functions]
{Extreme values of quadratic Hecke
$L$-functions over global fields}

\author{Zikang Dong}\author{Long Liu}
\date{\today}
\address[Zikang Dong]{School of Mathematical Sciences, Soochow University, Suzhou, 215006, P. R. China}
\email{zikangdong@gmail.com}
\address[Long Liu]{Institute for Theoretical Sciences, Westlake University, Hangzhou,  310024, P. R. China}
\email{liulong@westlake.edu.cn}

\begin{document}
\raggedbottom

\begin{abstract}
We study large values of quadratic Hecke $L$-functions in the conductor
aspect. Let $K$ be a fixed number field, and assume GRH for its finite-order
Hecke $L$-functions. In a fixed ray class component with conductor norm
comparable to $X$, we prove that
\[
 \max_\chi L\left(\frac12+\frac A{\log_2X},\chi\right)
 \geq\exp\left\{(e^{-A}+o(1))
       \sqrt{\frac{\log X\log_3X}{\log_2X}}\right\}
\]
for every fixed $A\geq0$. Every fixed smaller constant is attained by at
least $X^{1-o(1)}$ characters. The same count holds at a suitably slowly
moving threshold approaching the displayed constant. The combinatorial
input is a sparse squarefree G\'al set of cardinality $N$, retaining the
known leading constant $2$ and having square multiplicative energy
$N^{2+o(1)}$. The energy bound reflects the low degree of the associated
Boolean polynomial. Together with the resonance estimate, it yields the
abundance bound.
We also prove
unconditional analogues for quadratic characters with prime conductor
away from one fixed place over any global function field of odd
characteristic. Finally, we give bounds in the fixed strip and at $s=1$,
including the dependence on the residue of the Dedekind zeta function
and the prescribed local factors.
\end{abstract}

\maketitle
\tableofcontents

\section{Introduction}

We study large values of quadratic Hecke $L$-functions near the
central point, with particular attention to how often such values occur.
A lower bound for the maximum obtained by resonance need not, by itself,
give a useful count of the characters attaining large values: the
resonating weight may be concentrated on too few members of the family.
Here we use sparse squarefree resonators whose fourth moments can also
be controlled. This allows us to prove that every fixed lower threshold
in our maximum estimate is reached by $X^{1-o(1)}$ characters in a family
of size comparable to $X$. The argument applies at the centre and at
real points approaching it from the right.

Let $K$ be a fixed number field of degree $d_K$, with ring of integers
$\OO_K$, and let $\Norm\aideal$ denote the absolute norm of a nonzero
integral ideal. Unless stated otherwise, ideal sums are over nonzero
integral ideals of $\OO_K$. We write $\log_j$ for the $j$-fold iterated
natural logarithm and set
\[
 \mathcal V(X)=\sqrt{\frac{\log X\log_3X}{\log_2X}}.
\]
The family is chosen in Proposition~\ref{prop:positive-local-component}.
It consists of squarefree ideals $\aideal$ in the identity class of a
fixed narrow ray modulus with finite part $\q$. Denote this set by
$\Om$ and the finite support of $\q$ by $S$. For each $\aideal\in\Om$,
choose a totally positive generator $\alpha_{\aideal}\equiv1\pmod\q$.
With a suitable fixed totally positive $d\in K^\times$, the extension
$K(\sqrt{d\alpha_{\aideal}})/K$ defines a quadratic Hecke character
$\chi_{\aideal}$ independent of this choice. Its conductor satisfies
\[
 \fideal_{\aideal}=\cideal_{\Om}\aideal,\qquad
 \chi_{\aideal}(\p)=0\quad(\p\in S),
\]
where $\cideal_{\Om}$ is fixed. As usual, the ideal coefficients are
extended by zero at ramified primes. Thus the conductor norm is
comparable to $\Norm\aideal$. We write
\[
 \Om(X)=\{\aideal\in\Om:X<\Norm\aideal\leq2X\}.
\]
Then $\#\Om(X)\sim\kappa_{\Om}X$ for a constant $\kappa_{\Om}>0$.

Throughout the number field arguments, \textup{GRH}$_K$ denotes GRH
for all finite-order Hecke $L$-functions over $K$. The larger class of
characters is needed in the reciprocity argument: extending a quadratic
character from a subgroup of a ray class group need not preserve its
order. For $A\geq0$, put
\begin{equation}\label{eq:sigma-A}
 \sigma_A(X)=\frac12+\frac A{\log_2X}.
\end{equation}

\begin{theorem}\label{thm:central-max}Assume \textup{GRH}$_K$. Then the following assertions hold.
\begin{enumerate}
\item For every fixed $A\geq0$,
\[
 \max_{\aideal\in\Om(X)}L(\sigma_A(X),\chi_{\aideal})
 \geq\exp\{(e^{-A}+o(1))\mathcal V(X)\}.
\]
\item For every fixed $A\geq0$ and $0<c<1$,
\[
 \#\{\aideal\in\Om(X):L(\sigma_A(X),\chi_{\aideal})
                 \geq e^{ce^{-A}\mathcal V(X)}\}\geq X^{1-o(1)}.
\]
\item For every fixed $A_0\geq0$, there is a function
$\eta_{A_0}(X)\to0$ such that, uniformly for $0\leq A\leq A_0$,
\[
 \#\{\aideal\in\Om(X):L(\sigma_A(X),\chi_{\aideal})
       \geq e^{(1-\eta_{A_0}(X))e^{-A}\mathcal V(X)}\}
       \geq X^{1-o(1)}.
\]
\end{enumerate}
The first two assertions are also uniform for $A$ in a fixed compact interval.
\end{theorem}

At $A=0$ this gives the corresponding statements at the centre.
The count $X^{1-o(1)}$ means that, for each fixed $\varepsilon>0$,
there are at least $X^{1-\varepsilon}$ such characters for all sufficiently
large $X$. This does not imply a positive proportion. The threshold in (iii)
is obtained by diagonal selection, with no prescribed rate for
$\eta_{A_0}(X)\to0$.

Over $\mathbb Q$ the result can be stated in a single congruence class
of positive fundamental discriminants. Write
$\chi_D(n)=(D/n)$ for the Kronecker character associated with a positive
fundamental discriminant $D$, and put
\[
 \mathcal D_8(X)=\{D:X<D\leq2X,\ D=8a,\ a\text{ squarefree},\quad
                         a\equiv1\pmod8\}.
\]
These are precisely the positive fundamental discriminants in the
congruence class $8\pmod{64}$ and the indicated interval.

\begin{corollary}\label{cor:rational-abundance}
Assume GRH for Dirichlet $L$-functions. For every fixed $A\geq0$
and $0<c<1$,
\[
 \#\{D\in\mathcal D_8(X):L(\sigma_A(X),\chi_D)
                  \geq e^{ce^{-A}\mathcal V(X)}\}\geq X^{1-o(1)}.
\]
Moreover,
\[
 \max_{D\in\mathcal D_8(X)}L(\sigma_A(X),\chi_D)
       \geq\exp\{(e^{-A}+o(1))\mathcal V(X)\}.
\]
Both assertions are uniform for $A$ in a fixed compact subset of
$[0,\infty)$, and remain valid with $\mathcal D_8(X)$ replaced by all
positive fundamental discriminants in $(X,2X]$.
\end{corollary}

\begin{proof}
For $K=\mathbb Q$, take $S=\{2\}$, $d=2$, and $\q=(8)$ in the
family construction. The component consists of squarefree positive
integers $a\equiv1\pmod8$. The extension $\mathbb Q(\sqrt{2a})$
has discriminant and conductor $8a$, and its character is $\chi_{8a}$.
To express the result in the conductor variable, set
\[
 Y=X/8,\qquad B=A\frac{\log_2Y}{\log_2X}.
\]
Then $\sigma_B(Y)=\sigma_A(X)$ exactly, and, uniformly for bounded $A$,
\[
 e^{-B}\mathcal V(Y)=(1+o(1))e^{-A}\mathcal V(X).
\]
Apply Theorem~\ref{thm:central-max}(ii) with any fixed $c<c'<1$ and
use its compact-parameter uniformity. For sufficiently large $X$, the
threshold with $c'$ exceeds the one stated here, and
$Y^{1-o(1)}=X^{1-o(1)}$. The maximum follows from assertion (i) of
the same theorem. Inclusion gives the claims for all positive
fundamental discriminants.
\end{proof}

The analogous result over function fields is unconditional.
Let $F=\mathbb F_q(C)$, where $q$ is odd and $C$ is a smooth, projective,
geometrically connected curve, and fix an odd-degree place $\infty$.
We use the component $\Om_F$ of principal primes in
$\Gamma(C\setminus\{\infty\},\mathcal O_C)$ constructed in
Section~\ref{sec:function-field}. A normalized generator of
$\mathfrak P\in\Om_F$ defines a quadratic character
$\chi_{\mathfrak P}$ with conductor divisor $\mathfrak P+\infty$.
The finite conductor away from $\infty$ is therefore prime.
For the admissible degrees, write
\[
 \Om_F(n)=\{\mathfrak P\in\Om_F:\deg\mathfrak P=n\},\qquad X_n=q^n,
\]
and set
\begin{equation}\label{eq:VF}
 \mathcal V_F(n)=\sqrt{\frac{n\log q\,\log_2n}{\log n}}
                  =(1+o(1))\mathcal V(X_n),
\end{equation}
\begin{equation}\label{eq:ff-sigma-A}
 \sigma_{A,n}=\frac12+\frac A{\log_2X_n}.
\end{equation}

\begin{theorem}\label{thm:function-field}
Fix $F=\mathbb F_q(C)$ with $q$ odd, and let $n$ tend to infinity through
the admissible degrees of $\Om_F$.
\begin{enumerate}
\item For every fixed $A\geq0$,
\[
 \max_{\mathfrak P\in\Om_F(n)}L(\sigma_{A,n},\chi_{\mathfrak P})
 \geq\exp\{(e^{-A}+o(1))\mathcal V_F(n)\}.
\]
\item For every fixed $A\geq0$ and $0<c<1$,
\[
 \#\{\mathfrak P\in\Om_F(n):L(\sigma_{A,n},\chi_{\mathfrak P})
                \geq e^{ce^{-A}\mathcal V_F(n)}\}\geq X_n^{1-o(1)}.
\]
\item For every fixed $A_0\geq0$, there is a function
$\eta_{F,A_0}(n)\to0$ such that, uniformly for $0\leq A\leq A_0$,
\[
 \#\{\mathfrak P\in\Om_F(n):L(\sigma_{A,n},\chi_{\mathfrak P})
       \geq e^{(1-\eta_{F,A_0}(n))e^{-A}\mathcal V_F(n)}\}
       \geq X_n^{1-o(1)}.
\]
\end{enumerate}
The first two assertions are uniform on fixed compact intervals of $A$.
No unproved hypothesis is required.
\end{theorem}

The restriction to prime conductors gives a twisted character average
with error $O_F((1+\deg\mathfrak l_0)q^{n/2}/n)$, where
$\mathfrak l_0$ is the squarefree part of the twisting ideal.
Its linear dependence on $\deg\mathfrak l_0$ is useful here, since the
resonator contains divisors of degree much larger than $n$.

We also obtain number field bounds at fixed points to the right of
the centre.
For $0<b<1$ and $1/2<\sigma<1$, put
\begin{equation}\label{eq:alpha-def}
 \alpha(b)=2\log\frac{(1+b)^2}{1+b^2},\qquad
 \widetilde\alpha(\sigma,b)=
 \frac{2b}{(1+b^2)(1-\sigma)}\alpha(b)^{\sigma-1}.
\end{equation}
The factor $2b/(1+b^2)$ comes from the local square product count.
$\alpha(b)$ determines the permitted length of the resonator.

\begin{theorem}\label{thm:strip}
Assume \textup{GRH}$_K$, and fix $1/2<\sigma<1$ and $0<b<1$.
\begin{enumerate}
\item One has
\[
 \max_{\aideal\in\Om(X)}\log L(\sigma,\chi_{\aideal})
 \geq(\widetilde\alpha(\sigma,b)+o(1))
               (\log X)^{1-\sigma}(\log_2X)^{-\sigma}.
\]
\item For $0<\eta<1/\alpha(b)$, define
\[
 \tau_{\sigma,\eta}(X)=
 \widetilde\alpha(\sigma,b)(1-\eta\alpha(b))^{1-\sigma}
               (\log X)^{1-\sigma}(\log_2X)^{-\sigma}.
\]
Then
\[
 \frac{\#\{\aideal\in\Om(X):
       \log L(\sigma,\chi_{\aideal})>\tau_{\sigma,\eta}(X)\}}
 {\#\Om(X)}\geq X^{-\frac12(1-\eta\alpha(b))+o(1)}.
\]
\end{enumerate}
\end{theorem}

At $s=1$ the leading multiplier includes the residue
$\kappa_K=\Res_{s=1}\zeta_K(s)$ and the factors at the fixed ramified
primes. With $\gamma$ denoting Euler's constant, put
\begin{equation}\label{eq:MKOmega}
 \mathcal M_{K,\Om}=e^\gamma\kappa_K
                       \prod_{\p\in S}(1-(\Norm\p)^{-1}).
\end{equation}
We shall also use the constant
\begin{equation}\label{eq:C2}
 C_2=\frac\pi4-\frac{\log2}{2}
       +\log\left(2\left(3\log2-\frac\pi2\right)\right)
       =0.455967\ldots.
\end{equation}

\begin{theorem}\label{thm:one-line}
Assume \textup{GRH}$_K$.
\begin{enumerate}
\item One has
\[
 \max_{\aideal\in\Om(X)}L(1,\chi_{\aideal})
 \geq\mathcal M_{K,\Om}(\log_2X+\log_3X-C_2+o(1)).
\]
\item For every fixed $\eta>0$, the proportion of $\aideal\in\Om(X)$ with
\[
 L(1,\chi_{\aideal})>
 \mathcal M_{K,\Om}(\log_2X+\log_3X-C_2-\eta)
\]
is at least $X^{-e^{-\eta}/2+o(1)}$.
\end{enumerate}
\end{theorem}

\subsection{Earlier work}

Soundararajan introduced the resonance method for extreme values of zeta
and $L$-functions \cite[Theorems~1--2]{Sound2008}. Its connection with
GCD sums was developed by Aistleitner \cite{Aistleitner2016} and by
Bondarenko and Seip \cite{BondarenkoSeipDuke,BondarenkoSeipMathAnn}.
Our combinatorial construction follows the sparse prime blocks of
de la Bret\`eche and Tenenbaum \cite[Section~2.2]{BretTenenbaum} and
retains the leading constant in their squarefree G\'al estimate
\cite[equation~(1.5), Remark after Theorem~1.1]{BretTenenbaum}.

For quadratic Dirichlet $L$-functions, Darbar and Maiti
\cite[Theorem~1]{DarbarMaiti} proved, under GRH, a central maximum
lower bound with coefficient $1/2$ on the scale $\mathcal V(X)$.
Dong, Wang, Zhang and Zhao \cite[Theorem~1.1]{DongWangZhangZhao}
obtained coefficient $1$. At the moving point $1/2+A/\log_2X$,
Dong, Li, Song and Zhao \cite[Theorem~1.1]{DongLiSongZhao} obtained
$e^{-A}/(2(e-1))$. Theorem~\ref{thm:central-max} gives $e^{-A}$
and counts the characters exceeding every fixed smaller threshold.
At the centre over $\mathbb Q$, the maximum coefficient is thus the
one already known. The additional conclusion is the abundance bound,
which holds even in the congruence class of
Corollary~\ref{cor:rational-abundance}. Over a fixed number field,
the theorem gives both assertions in the component $\Om$.
These are lower bounds on the resonance scale, rather than asymptotic
formulas for the true maximum. See \cite{FarmerGonekHughes} for
conjectures on extreme values.

The point of the construction is to retain the G\'al gain while keeping
the fourth moment small. Proposition~\ref{prop:gal} provides one set
with the required cardinality, prime support, shifted G\'al sum and
square energy. Lemma~\ref{lem:gal-truncation} shows that its gain
survives the truncation in the approximate functional equation.
The energy estimate follows from sparsity, either by the counting
argument in Lemma~\ref{lem:hamming-energy} or by Bonami
hypercontractivity, as noted in Remark~\ref{rem:hypercontractive-energy}.
For function fields we use blocks of prescribed prime degree, where
the discreteness of the norms must be taken into account.

Table~\ref{tab:central-comparison} summarizes these comparisons.
A coefficient $C$ denotes a maximum lower bound
$\exp\{(C+o(1))\mathcal V(X)\}$, with $\mathcal V_F(n)$ in place
of $\mathcal V(X)$ over function fields. The abundance entries refer
to every fixed coefficient $cC$, $0<c<1$. The number field bounds
assume GRH, whereas the function field bounds are unconditional.
Only the indicated maximum and abundance statements are compared.

\begin{table}[htbp]
\centering
\small
\setlength{\tabcolsep}{3pt}
\renewcommand{\arraystretch}{1.2}
\caption{Central and moving-point lower bounds on the resonance scale.}
\label{tab:central-comparison}
\begin{tabular}{@{}>{\raggedright\arraybackslash}p{.30\textwidth}
                  >{\raggedright\arraybackslash}p{.27\textwidth}
                  >{\centering\arraybackslash}p{.16\textwidth}
                  >{\raggedright\arraybackslash}p{.21\textwidth}@{}}
\toprule
Source & Family and point & Coefficient & Result compared \\
\midrule
Darbar--Maiti \cite{DarbarMaiti}
 & Quadratic Dirichlet, $A=0$ & $1/2$ & Maximum \\
Dong et al.\ \cite{DongWangZhangZhao}
 & Quadratic Dirichlet, $A=0$ & $1$ & Maximum \\
Dong et al.\ \cite{DongLiSongZhao}
 & Quadratic Dirichlet, $A>0$ & $\dfrac{e^{-A}}{2(e-1)}$ & Maximum \\
Theorem~\ref{thm:central-max}; Corollary~\ref{cor:rational-abundance}
 & Fixed number field, $A\geq0$ & $e^{-A}$ & Maximum and $X^{1-o(1)}$ count \\
\midrule
Darbar--Maiti \cite{DarbarMaitiFF}
 & Prime polynomials over $\mathbb F_q(t)$, $A=0$ & $1/2$ & Maximum \\
Theorem~\ref{thm:function-field}
 & Principal primes over a fixed curve, $A\geq0$ & $e^{-A}$
 & Maximum and $X_n^{1-o(1)}$ count \\
\bottomrule
\end{tabular}
\end{table}

Related results with restrictions on rational conductors appear in
Fan, Hua and Xie \cite{FanHuaXie}, Gao \cite{GaoPrime}, and
Dong, Wang, Zhang and Zhao \cite{DongWangZhangZhaoPrime}.
For number fields, the quadratic Hecke families of Goldmakher and Louvel
\cite[Definition~1 and Section~2]{GoldmakherLouvel} provide the relevant
reciprocity framework. We work in an explicit Kummer component and
prove the character average needed for that component directly.
Moment results over imaginary quadratic fields may be found in
\cite{GaoZhaoMoments,GaoZhaoFirstMoment}. Other families of Hecke
$L$-functions include the angular characters considered by White
\cite{White} and the class-group characters considered by
Campos-Vargas \cite{CamposVargas}.

Over $\mathbb F_q(t)$, Darbar and Maiti
\cite[Theorem~1]{DarbarMaitiFF} obtained coefficient $1/2$ for the
central maximum in a prime-polynomial family. Theorem~\ref{thm:function-field}
gives coefficient $1$, together with abundance and moving point bounds,
over any fixed curve of odd characteristic. We discuss the conventions
for polynomial quadratic symbols in Section~\ref{sec:function-field}.
For large values and value distributions in function fields, see
\cite{DjokicLelasVrecica,LumleyStrip,LumleyOne}. For central moments
and limit theorems, see \cite{AndradeKeating,Jung,DarbarLumley}.

Theorems~\ref{thm:strip} and \ref{thm:one-line} use the short
Euler product approach of \cite{AistleitnerMahatabMunschPeyrot} and
the quadratic resonators of \cite{DarbarMaiti}. They give number field
versions with the local factors kept explicit. Related results in the
strip appear in \cite{Lamzouri,Lamzouri2026}, and results at $1$ in
\cite{MontgomeryVaughan,GranvilleSoundararajan,AistleitnerMahatabMunsch}.

\subsection{Outline of the proof}

For a set $\mathcal M$ of squarefree ideals, consider
\[
 R_\chi=\sum_{\mideal\in\mathcal M}\chi(\mideal).
\]
We estimate the first moment of $L(\sigma,\chi)$ with weight
$|R_\chi|^2$ and compare it with the second moment of the resonator.
After applying the approximate functional equation and averaging over
the characters, square products give the main term. Given
$\mideal,\nideal\in\mathcal M$, the ideal
$[\mideal,\nideal]/(\mideal,\nideal)$ makes their product a square.
Here parentheses and brackets denote the ideal gcd and lcm. Its
contribution to the first moment contains the kernel
\[
 \left(\frac{\Norm(\mideal,\nideal)}{\Norm[\mideal,\nideal]}\right)^\sigma.
\]
The first moment lower bound is therefore governed by a G\'al sum.

The fourth moment involves a different count:
$E_\square(\mathcal M)$, the number of ordered quadruples in
$\mathcal M$ whose product is a square. For $|\mathcal M|=N$,
we require $E_\square(\mathcal M)\leq N^{2+o(1)}$ as well as the
G\'al lower bound. Both estimates are obtained by taking products of
sparse subsets of disjoint prime blocks. Once the fourth moment has
been bounded, Cauchy--Schwarz and the individual estimate
$|L(\sigma,\chi)|\leq X^{o(1)}$ give the abundance assertion.

For each fixed $\varepsilon>0$, the error in the number field weighted
first moment is $O(X^{3/4+\varepsilon+o(1)}N^2)$, compared with the
main scale $XN$. We take $N=X^\beta$ with fixed $\beta<1/4$ and
$\varepsilon<1/4-\beta$. The G\'al constant $2$ then gives
$2\sqrt\beta$ in the maximum bound. Letting $\beta\uparrow1/4$
gives coefficient $1$. At the moving point, the primes in the blocks
satisfy $\log\Norm\p=(1+o(1))\log_2N$, and
$\log_2N/\log_2X\to1$. The shift in each prime weight thus
contributes $e^{-A}+o(1)$, uniformly for bounded $A$.

Sections~\ref{sec:family}--\ref{sec:analytic} establish the family,
the character average and the analytic estimates.
The G\'al construction and its truncation are proved in
Section~\ref{sec:gal}, and Theorem~\ref{thm:central-max} follows in
Section~\ref{sec:central}. Section~\ref{sec:function-field} treats
function fields. The fixed-strip and $1$-line arguments occupy
Sections~\ref{sec:strip} and \ref{sec:one}.

Throughout, $O(\cdot)$, $\ll$, and $o(\cdot)$ refer to the parameter
tending to infinity. Implicit constants may depend on the fixed field,
the fixed component, and explicitly fixed auxiliary parameters.
When a parameter varies over a compact interval, we state the required
uniformity. We use $\one_E$ for the indicator of a condition or set $E$.

\section{The quadratic family}\label{sec:family}

We construct a family with two properties: its conductors vary with
squarefree ideals, and its fixed local factors contribute no signs to the
character averages.
Starting from a fixed quadratic extension, we multiply its defining
square class by generators of suitable principal ideals. The construction
is a restricted form of that in Goldmakher--Louvel
\cite[Section~2]{GoldmakherLouvel}, with the local conditions specified
for the averages below.

Let $\mathbb A_K$ denote the adele ring of $K$. A finite-order Hecke
character is a continuous character
\[
 \omega:K^\times\backslash\mathbb A_K^\times\longrightarrow\mathbb C^\times
\]
with finite image. Its finite conductor is denoted by $\fideal_\omega$.
We say that its infinite type is trivial if its archimedean components are
trivial. Its associated ideal character is
\[
 \chi_\omega(\nideal)=
 \begin{cases}
 \displaystyle\prod_{\p^e\parallel\nideal}
       \omega_\p(\varpi_\p)^e,&(\nideal,\fideal_\omega)=1,\\[4pt]
 0,&(\nideal,\fideal_\omega)>1,
 \end{cases}
\]
where $\varpi_\p$ is a uniformizer of $K_\p$. At an unramified prime this
value is independent of the uniformizer. The zero in this definition
records ramification. A local character itself never vanishes. We use
$L(s,\chi_\omega)$ and $L(s,\omega)$ for the same Hecke $L$-function,
defined for $\Re s>1$ by the absolutely convergent expressions
\[
 L(s,\chi_\omega)=\sum_{\nideal}
       \frac{\chi_\omega(\nideal)}{(\Norm\nideal)^s}
 =\prod_{\p}\left(1-\frac{\chi_\omega(\p)}{(\Norm\p)^s}\right)^{-1}.
\]
Here $K_\p$ is the completion at $\p$. We write $v_\p$ for its
normalized valuation.
See \cite[Section~3.8]{IK} for ideal and Hecke characters.

A narrow ray modulus will always include every real place. We denote its
finite ideal part by $\q$. If $S$ is a finite set of finite primes, the
notation $(\aideal,S)=1$ means that no prime in $S$ divides $\aideal$.

\begin{proposition}\label{prop:positive-local-component}
There are a finite set $S$ containing the dyadic primes, a narrow ray
modulus with finite part $\q$ supported on $S$, and an ideal $\cideal_+$
supported on $S$, with the following properties.
\begin{enumerate}[label=\textup{(\roman*)}]
\item Let $\Om_+$ be the squarefree integral ideals in the identity narrow
ray class modulo $\q$. Each $\aideal\in\Om_+$ determines a primitive
quadratic character $\omega_\aideal$ of trivial infinite type, with
\begin{equation}\label{eq:positive-local-conductor}
 \fideal_\aideal=\cideal_+\aideal,
 \qquad \chi_\aideal(\p)=0\quad(\p\in S).
\end{equation}
Distinct ideals give distinct characters.
\item The set $\Om_+$ has positive density among integral ideals. If
$\mathrm{Cl}_\q^+$ is the narrow ray class group, its hat denotes the character group, and $\mu_K$ is
the ideal M\"obius function, then
\begin{equation}\label{eq:component-fourier}
 \one_{\Om_+}(\aideal)
 =\frac{\mu_K^2(\aideal)\one_{(\aideal,S)=1}}
        {|\mathrm{Cl}_\q^+|}
   \sum_{\xi\in\widehat{\mathrm{Cl}_\q^+}}\xi(\aideal).
\end{equation}
\end{enumerate}
\end{proposition}

\begin{proof}
Different totally positive generators of a principal ideal can define
different quadratic extensions. We impose congruences that make every
remaining unit ambiguity a square.

Let $U_K^+$ denote the totally positive units. The quotient
$U_K^+/U_K^2$ is finite. For each nontrivial class choose a representative
$u$ and a distinct nondyadic prime $\mathfrak t_u$ at which $u$ is not a
square. Such primes exist by Chebotarev applied to $K(\sqrt u)/K$.
Let $T$ consist of these primes and the dyadic primes.

Choose $\pi_\p\in\p\setminus\p^2$ for each $\p\in T$.
The Chinese remainder theorem gives an integral $d_0$ with
$d_0\equiv\pi_\p\pmod{\p^2}$ for all $\p\in T$.
Choose a positive rational integer
$M\in\prod_{\p\in T}\p^2$. Such an $M$ exists because every nonzero
integral ideal of $\OO_K$ has nonzero intersection with $\mathbb Z$.
For a sufficiently large positive integer $k$, the element $d=d_0+kM$
is totally positive and has valuation one at every prime of $T$. Put
\[
 S=\{\p:v_\p(d)\text{ is odd}\}.
\]
Then $S$ contains $T$. In particular, every dyadic prime belongs to $S$.
The quadratic extension $K(\sqrt d)/K$ is ramified at every prime in $S$.
Outside $S$ the residue characteristic is therefore odd. Moreover,
$v_\p(d)$ is even, so the extension is unramified there. Write
$\psi_d$ for its quadratic character and $\cideal_+$ for its finite
conductor. Thus $\cideal_+$ has support exactly $S$.

For each $\p\in S$ choose $M_\p$ so large that
\begin{equation}\label{eq:deep-local-squares}
 1+\p^{M_\p}\OO_{K,\p}\subset K_\p^{\times2}.
\end{equation}
Set $\q=\prod_{\p\in S}\p^{M_\p}$, increasing the exponents if necessary
so that $\cideal_+\mid\q$. These conditions imply
\begin{equation}\label{eq:ray-units-square}
 \{u\in U_K^+:u\equiv1\pmod\q\}\subset U_K^2.
\end{equation}
Indeed, a nonsquare class has a chosen prime $\mathfrak t_u$ where it is
not a local square. Equation~\eqref{eq:deep-local-squares} excludes that
class from the left side.

For $\aideal\in\Om_+$ choose a generator satisfying
\[
 (\alpha_\aideal)=\aideal,
 \qquad \alpha_\aideal>0\text{ at every real place},
 \qquad \alpha_\aideal\equiv1\pmod\q.
\]
Define $\omega_\aideal$ to be the quadratic character of
$K(\sqrt{d\alpha_\aideal})/K$. Any two permitted generators differ by a
unit in \eqref{eq:ray-units-square}. Hence the character is well defined.
It is nontrivial because $d\alpha_\aideal$ has odd valuation at every
prime in $S$.

At $\p\in S$, the generator $\alpha_\aideal$ is a local square.
The local character therefore agrees with $\psi_d$ and has the same
conductor exponent. Since $\aideal$ is a ray class ideal, it is coprime
to $\q$, hence no prime of $S$ divides $\aideal$. At a prime dividing
$\aideal$, the residue characteristic is therefore odd and
$d\alpha_\aideal$ has odd valuation. The character is tamely ramified
with conductor exponent one. At every remaining finite prime the residue
characteristic is odd and the valuation of $d\alpha_\aideal$ is even.
The corresponding quadratic extension is unramified (possibly split).
Total positivity gives trivial infinite type. This proves
\eqref{eq:positive-local-conductor}. The conductor recovers $\aideal$,
so distinct indices give distinct characters.

Character orthogonality in $\mathrm{Cl}_\q^+$ gives
\eqref{eq:component-fourier}. We verify its density assertion by a
Dirichlet series. For a ray character $\xi$, write $L^S(s,\xi)$ for its
Hecke $L$-function with the Euler factors in $S$ removed. Then
\begin{equation}\label{eq:cell-dirichlet-series}
 \sum_{(\aideal,S)=1}
 \frac{\mu_K^2(\aideal)\xi(\aideal)}{(\Norm\aideal)^s}
 =\prod_{\p\notin S}\left(1+\frac{\xi(\p)}{(\Norm\p)^s}\right)
 =\frac{L^S(s,\xi)}{L^S(2s,\xi^2)}.
\end{equation}
The denominator is nonzero for $\Re s>1/2$ by its absolutely convergent
Euler product. Only the principal $\xi$ contributes a pole at $s=1$.
After division by $|\mathrm{Cl}_\q^+|$, its residue is
\[
 \kappa_{\Om_+}
 =\frac{\Res_{s=1}\zeta_K(s)}
 {|\mathrm{Cl}_\q^+|\zeta_K(2)}
 \prod_{\p\in S}\left(1+\frac1{\Norm\p}\right)^{-1}>0.
\]
For a fixed smooth compactly supported function $W$ on $(0,\infty)$,
Mellin inversion and a shift to $\Re s=1/2+\eps$ now give
\begin{equation}\label{eq:cell-count}
 \sum_{\aideal\in\Om_+}W(\Norm\aideal/X)
 =\kappa_{\Om_+}X\int_0^\infty W(t)\,dt
  +O_{K,\q,W,\eps}(X^{1/2+\eps}).
\end{equation}
Here the fixed Hecke $L$-functions have polynomial growth on vertical
lines, and the Mellin transform of $W$ decays faster than every power.
No hypothesis on zeros is needed for this smoothed count.
Smooth majorants and minorants of $\one_{[1,2]}$ can have integrals
arbitrarily close to one. Applying the last estimate to them gives
$\#\{\aideal\in\Om_+:X<\Norm\aideal\leq2X\}
\sim\kappa_{\Om_+}X$, proving positive density.
\end{proof}

The use of a ramified fixed character is deliberate. For example, positive
fundamental discriminants $d\equiv5\pmod8$ satisfy $\chi_d(2)=-1$.
A fixed congruence class alone would therefore not give nonnegative local
terms in the averages below.

We henceforth fix the component $\Om=\Om_+$ and put
$\cideal_\Om=\cideal_+$. We refer to it as the positive component.
The property we shall use is
\begin{equation}\label{eq:bad-prime-terms-vanish}
 (\kideal,S)>1\quad\Longrightarrow\quad
 \chi_\aideal(\kideal)=0\qquad(\aideal\in\Om).
\end{equation}
The construction gives a fixed twist of the principal-ideal characters
in \cite[Section~2]{GoldmakherLouvel}. We need no extension to a family
indexed by every squarefree ideal.

We use $\kappa_\Om$ for the constant in \eqref{eq:cell-count}.
Under \textup{GRH}$_K$ the sharp count has the same error exponent:
\begin{equation}\label{eq:cell-sharp}
 \#\Om(X)=\kappa_\Om X+O_{K,\Om,\eps}(X^{1/2+\eps}).
\end{equation}
To see this, apply truncated Perron inversion to
\eqref{eq:cell-dirichlet-series}. For every fixed $\delta,\eta>0$,
GRH bounds its numerator by
$O_{K,\q,\delta,\eta}((1+|t|)^\eta)$ uniformly for
$1/2+\delta\leq\Re s\leq2$ and $|t|\geq1$.
See \cite[Section~5.7]{IK}. For the principal character the pole at
$s=1$ is kept as a residue, so this bound is not asserted near that
pole. On the new vertical line $\Re s=1/2+\delta<1$, the compact
part $|t|\leq1$ is bounded separately. The inverse denominator is
bounded throughout the strip by absolute convergence.
The coefficient of $n^{-s}$ is bounded in modulus by the number of ideals
of norm $n$, which is at most $d_{d_K}(n)\ll_{K,\eta}n^\eta$.
Replace $Y$ by a half-integer with the same integer part, and take
the Perron height $T=Y^2$ and the initial line $\Re s=1+1/\log Y$.
This replacement preserves the summatory function, since ideal norms
are integers, and changes the residue main term by only $O_{K,\q}(1)$.
The truncation error is
$O_{K,\q,\eta}(Y^{1+\eta}/T)$, with a logarithm absorbed into $Y^\eta$.
Shifting to $\Re s=1/2+\delta$ contributes
$O(Y^{1/2+\delta}T^\eta\log T+Y^{1+\eta}T^{-1+\eta})$.
For $0<\eps<1/2$, take $\delta=\eps/4$ and $\eta=\eps/16$.
Then $\delta+2\eta<\eps$ and all these errors are
$O_{K,\q,\eps}(Y^{1/2+\eps})$. Larger $\eps$ follow from this range.
The principal pole gives the residue already computed, and subtracting
the counts at $2X$ and $X$ proves \eqref{eq:cell-sharp}.

For $(\nideal,S)=1$ put
\begin{equation}\label{eq:hK}
 h_K(\nideal)=\prod_{\p\mid\nideal}\frac{\Norm\p}{\Norm\p+1}.
\end{equation}
This factor is the density correction for requiring a squarefree index
to be coprime to $\nideal$.

\section{Averages of characters}\label{sec:orthogonality}

Square arguments contribute the main term in the character average.
For nonsquare arguments, reciprocity gives a nonprincipal character in
the family index, to which we can apply GRH. Since the argument ideal
need not be principal, we construct this reciprocal character on a ray
class group. Its conductor bound supplies the required uniformity.

Fix an integer $J_0>2d_K+10$. For a smooth function $w$ supported on
$[X,2X]$, define
\[
 \|w\|_{*,X}=\sum_{j=0}^{J_0}X^j\|w^{(j)}\|_\infty.
\]
This norm allows the weight to vary with terms in an approximate
functional equation.

\begin{proposition}\label{prop:orthogonality}
Assume \textup{GRH}$_K$. Let $w$ be smooth and supported on $[X,2X]$.
Write $\nideal=\nideal_0\nideal_1^2$, where $\nideal_0$ is squarefree,
and assume $(\nideal,S)=1$. For every fixed $0<\eps<1/2$,
\begin{equation}\label{eq:orthogonality}
 \begin{split}
 \sum_{\aideal\in\Om}w(\Norm\aideal)\chi_\aideal(\nideal)
 &=\kappa_\Om h_K(\nideal)\one_{\nideal=\square}
       \int_0^\infty w(t)\,dt\\
 &\quad+O_{K,\Om,\eps}\left(
 \|w\|_{*,X}X^{1/2+\eps}F_\eps(\nideal_0)G_\eps(\nideal_1)\right),
 \end{split}
\end{equation}
where the conductor and coprimality factors may be taken as
\begin{enumerate}[label=\textup{(\roman*)}]
\item $\displaystyle F_\eps(\nideal_0)
 =\exp\{C_{K,\eps}(\log(2+\Norm\nideal_0))^{1-\eps}\}$;
\item $\displaystyle G_\eps(\nideal_1)
 =\prod_{\p\mid\nideal_1}
     \left(1+(\Norm\p)^{-1/2-\eps/2}\right)$.
\end{enumerate}
\end{proposition}

\begin{proof}
Let $H$ be the kernel of the natural homomorphism
$\mathrm{Cl}_{\q\nideal_0}^+\to\mathrm{Cl}_\q^+$.
Each class in $H$ is represented by a principal ideal $(\alpha)$,
coprime to $\q\nideal_0$, with $\alpha$ totally positive and
$\alpha\equiv1\pmod\q$. Set
\begin{equation}\label{eq:rho-n0-definition}
 \rho_{\nideal_0}((\alpha))
 =\prod_{\p\mid\nideal_0}\left(\frac{\alpha}{\p}\right),
\end{equation}
where the factors are quadratic residue symbols in the residue fields.
The primes involved are nondyadic. Changing the generator multiplies it
by a square, by \eqref{eq:ray-units-square}. Changing the representative
within its $\q\nideal_0$-ray class multiplies it by an element congruent
to one at every prime of $\nideal_0$. Thus \eqref{eq:rho-n0-definition}
is a well-defined character of $H$.

A character of a subgroup of a finite abelian group extends to the whole
group. Choose such an extension to $\mathrm{Cl}_{\q\nideal_0}^+$ and
retain the notation $\rho_{\nideal_0}$. For $\nideal_0=(1)$ choose the
principal extension. The extension has finite order, but need not have
order two. This is why we assume GRH for all finite-order Hecke
characters.

Its finite conductor divides $\q\nideal_0$.  The extension to the narrow ray
class group may carry an archimedean sign type, but there are only finitely
many such types (depending on $K$). They are absorbed in the constants below
and do not affect the finite-conductor divisibility.  We also need ramification
at every $\p\mid\nideal_0$. Choose a totally positive $\alpha$ congruent to one
modulo $\q\prod_{\mathfrak l\mid\nideal_0,\,\mathfrak l\ne\p}\mathfrak l$
and to a nonzero nonsquare modulo $\p$. The Chinese remainder theorem,
followed by addition of a sufficiently large positive rational integer
in the modulus, gives such an $\alpha$.
The class of $(\alpha)$ becomes trivial when $\p$ is removed from the
ray modulus, whereas $\rho_{\nideal_0}((\alpha))=-1$.
Every extension is therefore ramified at $\p$, and
\[
 \nideal_0\mid\fideal_{\rho_{\nideal_0}}\mid\q\nideal_0.
\]

Let $\psi_{\alpha_\aideal}$ denote the character of
$K(\sqrt{\alpha_\aideal})/K$. For $(\aideal,\nideal)=1$ we have
$\omega_\aideal=\psi_d\psi_{\alpha_\aideal}$ and
\begin{equation}\label{eq:component-direct-reciprocity}
 \chi_\aideal(\nideal)
 =\psi_d(\nideal)\rho_{\nideal_0}(\aideal).
\end{equation}
Here $\psi_d(\nideal)$ denotes its ideal value. Since
$(\aideal,\nideal)=1$, every prime $\p\mid\nideal_0$ is unramified in
$K(\sqrt{\alpha_\aideal})/K$.  The unramified quadratic local Artin
character evaluates on a uniformizer as the quadratic residue symbol,
so
\[
 \psi_{\alpha_\aideal}(\p)
   =\left(\frac{\alpha_\aideal}{\p}\right).
\]
This identity is independent of whether one uses arithmetic or geometric
Frobenius, since the value is in $\{\pm1\}$ and hence equals its inverse.
The square part $\nideal_1^2$ contributes one. Hence
\[
 \psi_{\alpha_\aideal}(\nideal)
 =\prod_{\p\mid\nideal_0}
   \left(\frac{\alpha_\aideal}{\p}\right)
 =\rho_{\nideal_0}(\aideal),
\]
where the last equality follows because the class of $\aideal$ in
$\mathrm{Cl}_{\q\nideal_0}^+$ lies in $H$ (it is principal with a totally
positive generator congruent to $1$ modulo $\q$), and the chosen extension
$\rho_{\nideal_0}$ restricts to \eqref{eq:rho-n0-definition} on $H$.
This proves \eqref{eq:component-direct-reciprocity}. If
$(\aideal,\nideal)>1$, the character value is zero by
\eqref{eq:positive-local-conductor}. In the nonsquare case the factor
$\psi_d(\nideal)$ is independent of $\aideal$ and has absolute value one,
so it may be suppressed in the error estimate below. In the square case
it equals one.

Insert \eqref{eq:component-fourier}. For each ray character $\xi$ put
$\rho=\xi\rho_{\nideal_0}$ and let $T$ consist of $S$ and the primes
dividing $\nideal$. The required Dirichlet series is
\[
 D_{\xi,\nideal}(s)
 =\sum_{(\aideal,T)=1}
    \frac{\mu_K^2(\aideal)\rho(\aideal)}{(\Norm\aideal)^s}
 =\frac{L^T(s,\rho)}{L^T(2s,\rho^2)}.
\]
Here $L^T(s,\rho)$ denotes the primitive Hecke $L$-function inducing
$\rho$, with all Euler factors in $T$ omitted. The same convention is
used for $\rho^2$, whose primitive conductor can be smaller.
For $\Re s>1/2$, the Euler product for $1/L^T(2s,\rho^2)$ converges
absolutely and locally uniformly. Thus it is holomorphic there, and the
only possible pole of $D_{\xi,\nideal}$ in this half-plane comes from
a principal numerator at $s=1$. In particular, no zero-free estimate
on the line $\Re(2s)=1$ is needed. If $\rho^2$ is principal, then
\begin{align*}
 L^T(2s,\rho^2)&=\zeta_K(2s)
       \prod_{\p\in T}(1-(\Norm\p)^{-2s}),\\
 \frac1{L^T(2s,\rho^2)}&=O_{K,T}(s-1/2)
\end{align*}
near $s=1/2$. The pole of this denominator therefore gives a zero
of its reciprocal, not an additional residue. The contour below stays
strictly to the right of this point.
If $\nideal_0\ne(1)$, then $\rho$ is ramified at every prime dividing
$\nideal_0$: multiplication by $\xi$, whose conductor is supported on
$S$, cannot remove that ramification. Thus $\rho$ is nonprincipal.
If $\nideal_0=(1)$, then $\rho=\xi$, so only the principal ray character
can contribute a pole at one.

We estimate the same shifted integral in both cases. For a nonprincipal
primitive character inducing $\rho$, the standard Littlewood bound under
GRH in terms of the analytic conductor implies, at $s=1/2+\eps+it$,
\[
 \log |L(s,\rho)|
 \ll_{K,\eps}
 \frac{(\log C(\rho,t))^{1-2\eps}}
      {\log\log(3+C(\rho,t))},
\]
with the harmless convention that the denominator is bounded below by a
positive constant. Since
$C(\rho,t)\ll_{K,\Om}\Norm\fideal_\rho(3+|t|)^{d_K}$ and
$\fideal_\rho\mid\q\nideal_0$, enlarging $C_{K,\eps}$ gives the
coarser bound needed here,
\[
 |L(1/2+\eps+it,\rho)|\ll_{K,\Om,\eps}
 (1+|t|)F_\eps(\nideal_0).
\]
Indeed the conductor part is dominated by
$(\log(2+\Norm\nideal_0))^{1-\eps}$, while the height contribution is
$o_\eps(\log(2+|t|))$ as $|t|\to\infty$. It is therefore at most
$\log(2+|t|)+O_{K,\eps}(1)$, which justifies the displayed linear
factor rather than an unspecified power of the height. If $\rho$ is
principal, its primitive function is the fixed $\zeta_K$: the same bound
on this shifted line follows from the corresponding Littlewood estimate
at large height. The remaining compact part of the line avoids its pole
and is bounded with a constant depending only on $K,\eps$.
Its pole at $s=1$ is handled separately by a residue.
See \cite[Section~5.7]{IK}. The local parameters of these finite-order
Hecke functions have modulus at most one, so the required Ramanujan
bound is automatic.

To keep the dependence on $\nideal_1$ explicit, first omit only the primes
in $S$ and those dividing $\nideal_0$. The inverse denominator is uniformly bounded
by the convergent product $\prod_\p(1+(\Norm\p)^{-1-2\eps})$.
In the numerator, primes dividing $\nideal_0$ are ramified, and the
finitely many primes in $S$ cost a fixed constant. Removing an additional
prime $\p\mid\nideal_1$ not already in $S\cup\{\p:\p\mid\nideal_0\}$
multiplies this squarefree Euler product by
$(1+\rho(\p)(\Norm\p)^{-s})^{-1}$.
For all sufficiently large $P=\Norm\p$, depending only on $\eps$,
\[
 |1+\rho(\p)P^{-s}|^{-1}
 \leq\frac1{1-P^{-1/2-\eps}}
 \leq1+P^{-1/2-\eps/2}.
\]
Only finitely many smaller primes remain. Their product is bounded by a
constant depending on $K$ and $\eps$. Consequently
\[
 |D_{\xi,\nideal}(1/2+\eps+it)|
 \ll_{K,\Om,\eps}(1+|t|)
       F_\eps(\nideal_0)G_\eps(\nideal_1).
\]

Let $\widehat w(s)=\int_0^\infty w(u)u^{s-1}\,du$.
Integration by parts $J_0$ times, after the substitution $u=Xv$, gives
\[
 |\widehat w(\sigma+it)|
 \ll_{J_0}\|w\|_{*,X}X^\sigma(1+|t|)^{-J_0}
 \qquad(1/2\leq\sigma\leq2).
\]
For each fixed $\nideal$ and $\rho$, the standard vertical-strip
growth estimate for the numerator, and the absolutely convergent inverse
denominator, give
$D_{\xi,\nideal}(v+it)=O_{K,\nideal,\eps}((1+|t|)^{d_K+2})$
uniformly for $1/2+\eps\leq v\leq2$ and $|t|\geq1$.
The finitely many removed Euler factors are bounded throughout this
strip. Since $J_0>2d_K+10$, the horizontal integrals tend to zero.
Mellin inversion and a contour shift from $\Re s=2$ to
$\Re s=1/2+\eps$ therefore give the stated error, using the uniform
bound on the new vertical line. The constants used only to let the
horizontal integrals vanish need not be uniform in $\nideal$.
When $\nideal$ is a square, the principal residue from
\eqref{eq:cell-count} is multiplied by
\[
 \prod_{\p\mid\nideal}\left(1+\frac1{\Norm\p}\right)^{-1}
 =h_K(\nideal).
\]
In this case $\psi_d(\nideal)=1$. When $\nideal$ is not a square,
there is no pole. These are exactly the two main terms in
\eqref{eq:orthogonality}.
\end{proof}

The error factors permit very large ideals when their prime divisors are
small. The following form is convenient for the resonators below.

\begin{lemma}\label{lem:FG-uniform}
Fix $B_0>0$ and $0<\eps<1/2$. Suppose
$\nideal=\uideal\videal$ satisfies
\begin{enumerate}[label=\textup{(\roman*)}]
\item $\Norm\uideal\leq X^{B_0}$;
\item every prime divisor of $\videal$ has norm at most
$z\leq(\log X)^{1+\eps/4}$.
\end{enumerate}
Writing $\nideal=\nideal_0\nideal_1^2$ with $\nideal_0$ squarefree,
we have uniformly
\[
 \log\{F_\eps(\nideal_0)G_\eps(\nideal_1)\}
 \ll_{K,B_0,\eps}(\log X)^{1-\eps/2}=o(\log X).
\]
In particular the product is $X^{o(1)}$, with one bound for all such
$\uideal,\videal$, irrespective of the prime multiplicities in $\videal$.
\end{lemma}

\begin{proof}
The prime ideal theorem gives
$\sum_{\Norm\p\leq z}\log\Norm\p\ll_K z$. Hence
\begin{align*}
 \log\Norm\nideal_0&\leq B_0\log X+O_K(z),\\
 \log F_\eps(\nideal_0)
 &\ll_{K,B_0,\eps}(\log X+z)^{1-\eps}
 \ll_{K,B_0,\eps}(\log X)^{1-\eps/2}.
\end{align*}
The last inequality uses
$(1+\eps/4)(1-\eps)\leq1-\eps/2$.

Put $a=1/2+\eps/2$. The primes coming from $\videal$ contribute at most
\[
 \sum_{\Norm\p\leq z}(\Norm\p)^{-a}
 \ll_{K,\eps}\frac{z^{1-a}}{\log(2+z)}=o(\log X)
\]
to $\log G_\eps(\nideal_1)$.
For the primes dividing $\uideal$, split at $y=\log X$.
The primes of norm at most $y$ contribute
$O_{K,\eps}(y^{1-a}/\log y)$ by partial summation.
There are at most $B_0\log X/\log y$ prime divisors of norm greater than
$y$, since their product divides $\uideal$.
Their contribution is at most
\[
 \frac{B_0\log X}{\log y}\,y^{-a}
 =B_0\frac{(\log X)^{1-a}}{\log_2X}=o(\log X).
\]
Since $(1+\eps/4)(1-a)<1-\eps/2$, these estimates also bound
$\log G_\eps(\nideal_1)$ by $O_{K,B_0,\eps}((\log X)^{1-\eps/2})$.
Combining the bounds proves the asserted uniform estimate.
\end{proof}

\section{The approximate functional equation and bounds under GRH}
\label{sec:analytic}

The resonance argument uses two analytic facts. We need an approximate
functional equation whose weights are nonnegative. We also need a uniform
upper bound to turn a weighted lower bound into a count of characters.
We prove the central and shifted versions together.

Let $r_1$ and $r_2$ be the numbers of real and complex places of $K$.
For $\aideal\in\Om$, put
\[
 \mathcal Q_{\aideal}=|D_K|\Norm\fideal_{\aideal},\qquad
 \gamma_K(s)=\pi^{-r_1s/2}(2\pi)^{-r_2s}
                 \Gamma(s/2)^{r_1}\Gamma(s)^{r_2}.
\]
Here $D_K$ is the discriminant of $K$. The omitted constant in each complex
gamma factor cancels in every ratio below. Since the characters have trivial
infinite type, their completed functions are
\[
 \Lambda(s,\chi_{\aideal})=
 \mathcal Q_{\aideal}^{s/2}\gamma_K(s)L(s,\chi_{\aideal}).
\]
We record explicitly why this normalization has root number $+1$.
Put $L_{\aideal}=K(\sqrt{d\alpha_{\aideal}})$.  The conductor--discriminant
formula for the nontrivial quadratic Artin character gives
\[
 |D_{L_{\aideal}}|=|D_K|^2\Norm\fideal_{\aideal}.
\]
Moreover $d\alpha_{\aideal}$ is totally positive, so every real place of
$K$ splits in $L_{\aideal}$.  Consequently the quotient of the
archimedean factors of $\zeta_{L_{\aideal}}$ and $\zeta_K$ is precisely
$\gamma_K(s)$ (up to the harmless $s$-independent constant allowed by the
standard choice of complex gamma factor), while the quotient of the
conductor factors is $\mathcal Q_{\aideal}^{s/2}$.  Artin factorization
therefore identifies $\Lambda(s,\chi_{\aideal})$, up to such an
$s$-independent positive constant, with the quotient of the completed
Dedekind zeta functions of $L_{\aideal}$ and $K$.  Since both Dedekind
zeta functions satisfy functional equations with sign $+1$, the constant
cancels and
\[
 \Lambda(s,\chi_{\aideal})=\Lambda(1-s,\chi_{\aideal}).
\]
The character is nontrivial, and the completed primitive Hecke
$L$-function is entire.  Thus no pole is crossed in the contour shift
below except the pole of $1/z$ at $z=0$.  These are the standard
functional-equation facts for quadratic Hecke characters. See
\cite[Chapter~5]{IK}.

\begin{lemma}\label{lem:afe}Fix $0<\delta_0<1/4$. For $|\delta|\leq\delta_0$, define
$s_\delta=1/2+\delta$ and
\begin{equation}\label{eq:shifted-kernel}
 V_{K,\delta}(x)=\frac1{2\pi i}\int_{(c)}
 \frac{\gamma_K(s_\delta+z)}{\gamma_K(s_\delta)}x^{-z}\frac{dz}{z},
 \qquad c>0,
\end{equation}
where $\int_{(c)}$ means integration upward along $\Re z=c$.
The value is independent of the particular $c>0$ by contour shifting
inside the half-plane before the first gamma pole. Then the following
assertions hold.
\begin{enumerate}
\item For $0\leq\delta\leq\delta_0$,
\begin{align}
 L(s_\delta,\chi_{\aideal})
 &=\sum_{\kideal}
 \frac{\chi_{\aideal}(\kideal)}{(\Norm\kideal)^{s_\delta}}
 V_{K,\delta}\left(\frac{\Norm\kideal}{\sqrt{\mathcal Q_{\aideal}}}\right)
 \notag\\
 &\quad+Y_{\aideal,\delta}\sum_{\kideal}
 \frac{\chi_{\aideal}(\kideal)}{(\Norm\kideal)^{1-s_\delta}}
 V_{K,-\delta}\left(\frac{\Norm\kideal}{\sqrt{\mathcal Q_{\aideal}}}\right),
 \label{eq:shifted-afe}
\end{align}
where
\begin{equation}\label{eq:dual-factor}
 Y_{\aideal,\delta}=\mathcal Q_{\aideal}^{-\delta}
       \frac{\gamma_K(1/2-\delta)}{\gamma_K(1/2+\delta)}>0.
\end{equation}
\item The kernels take values in $[0,1]$. For
$0<\rho<1/2-\delta_0$, uniformly in $|\delta|\leq\delta_0$,
\begin{equation}\label{eq:shifted-kernel-small}
 V_{K,\delta}(x)=1+O_{K,\delta_0,\rho}(x^\rho),\qquad 0<x\leq1.
\end{equation}
For every integer $j\geq0$ and every $B\geq0$,
\begin{equation}\label{eq:shifted-kernel-decay}
 \left|(x\tfrac d{dx})^jV_{K,\delta}(x)\right|
       \ll_{K,\delta_0,j,B}(1+x)^{-B}.
\end{equation}
\end{enumerate}
In particular, if $V_K=V_{K,0}$, then
\begin{equation}\label{eq:afe}
 L(1/2,\chi_{\aideal})=2\sum_{\kideal}
 \frac{\chi_{\aideal}(\kideal)}{\sqrt{\Norm\kideal}}
 V_K\left(\frac{\Norm\kideal}{\sqrt{\mathcal Q_{\aideal}}}\right).
\end{equation}
\end{lemma}

\begin{proof}
The contour argument is the usual approximate functional equation
\cite[Section~5.2]{IK}. The central rational-field instance appears in
\cite[Lemmas~2.1--2.2, pp.~455--456]{Sound2000}.
We give the calculation because the positivity and uniformity of the shifted
weights will be used later.

Choose $c>1/2+\delta_0$ and set
\[
 I_\delta=\frac1{2\pi i}\int_{(c)}
 \frac{\Lambda(s_\delta+z,\chi_{\aideal})}
 {\mathcal Q_{\aideal}^{s_\delta/2}\gamma_K(s_\delta)}\frac{dz}{z}.
\]
Move the line to $\Re z=-c$. The gamma factors decay exponentially,
whereas the Hecke function has polynomial growth in a fixed vertical strip.
Thus the horizontal integrals vanish. Entireness leaves only the residue
at $z=0$. In the remaining integral substitute $z=-w$ and use the functional
equation. The reversed orientation gives
\begin{equation}\label{eq:afe-contour-shift}
 I_\delta=L(s_\delta,\chi_{\aideal})-Y_{\aideal,\delta}I_{-\delta}.
\end{equation}
Both Dirichlet series converge absolutely on the right-hand lines.
The dual Dirichlet series normally contains the contragredient character.
Here $\chi_{\aideal}$ is real quadratic, so
$\overline{\chi_{\aideal}}=\chi_{\aideal}$.  Expanding the two integrals
therefore gives \eqref{eq:shifted-afe}.

To see why the weights are positive, let $s\in[1/2-\delta_0,1/2+\delta_0]$.
Take independent random variables $U_j$ and $Z_j$ with respective densities
$t^{s/2-1}e^{-t}/\Gamma(s/2)$ and $t^{s-1}e^{-t}/\Gamma(s)$ on $t>0$.
Put
\[
 T_s=\prod_{j=1}^{r_1}(U_j/\pi)^{1/2}
            \prod_{j=1}^{r_2}(Z_j/(2\pi)).
\]
The gamma integral gives
$\mathbf E T_s^z=\gamma_K(s+z)/\gamma_K(s)$.
Mellin inversion, or Perron's formula applied inside this expectation, gives
\begin{equation}\label{eq:kernel-probability}
 V_{K,\delta}(x)=\mathbf P(T_{s_\delta}>x).
\end{equation}
The distributions are continuous. Consequently the identity holds for every
$x>0$, and $0\leq V_{K,\delta}(x)\leq1$.
Moreover,
\[
 1-V_{K,\delta}(x)
 \leq x^\rho\mathbf E T_{s_\delta}^{-\rho}\ll_{K,\delta_0,\rho}x^\rho,
\]
since $\rho<1/2-\delta_0$ keeps all gamma arguments positive.

For $x\geq1$, differentiate the Mellin integral and move its line to
$\Re z=B+1$. Stirling's formula proves \eqref{eq:shifted-kernel-decay}.
For $0<x\leq1$ and $j\geq1$, the factor $(-z)^j$ removes the pole at zero.
Move the line to $\Re z=-\rho$. No gamma pole is crossed.
The result is $O(x^\rho)$. For $j=0$, the probability bound suffices.
All constants are uniform on the stated interval of $\delta$.
\end{proof}

For the abundance argument, we need an individual bound with exponent
$o(\log X)$. The following estimate suffices.

\begin{lemma}\label{lem:central-upper}
Assume \textup{GRH}$_K$. For every fixed $A_0\geq0$, uniformly in
$\aideal\in\Om(X)$ and $0\leq\delta\leq A_0/\log_2X$,
\[
 |L(1/2+\delta,\chi_{\aideal})|
 \leq\exp\left(C_{K,\Om}\frac{\log X}{\log_2X}\right)=X^{o(1)}.
\]
\end{lemma}

\begin{proof}
We use the logarithmic majorant of Soundararajan
\cite[Proposition, p.~984, and Section~2]{SoundMoments}.
The following calculation verifies it for our Hecke functions and shifts.
Write $L(s)=L(s,\chi_{\aideal})$ and $\mathcal Q=\mathcal Q_{\aideal}$.
Under GRH the nontrivial zeros are $\rho=1/2+i\gamma$.
For real $u>1/2$, let
\[
 F(u)=\sum_\rho\frac{u-1/2}{(u-1/2)^2+\gamma^2}\geq0,
\]
where zeros are counted with multiplicity. Hadamard factorization gives
\begin{equation}\label{eq:hecke-hadamard-logder}
 -\Re\frac{L'}L(u)=\tfrac12\log\mathcal Q+O_K(1)-F(u)
 \qquad(1/2<u\leq1).
\end{equation}
Here the gamma logarithmic derivatives are bounded on the indicated real
interval, and the nonprincipal character contributes no pole.

Let $x\geq e^4$, $\Delta=1/\log x$, and $\sigma_0=\sigma+\Delta$,
where $1/2\leq\sigma\leq3/4$. For $a_0=\sigma_0-1/2\geq\Delta$, each zero
contributes
\[
 \int_\sigma^{\sigma_0}\frac{u-1/2}{(u-1/2)^2+\gamma^2}\,du
 =\frac12\log\frac{a_0^2+\gamma^2}{(a_0-\Delta)^2+\gamma^2}
 \geq\frac\Delta2\frac{a_0}{a_0^2+\gamma^2}.
\]
Indeed, $a_0^2-(a_0-\Delta)^2=\Delta(2a_0-\Delta)\geq\Delta a_0$,
and $\log(1+y)\geq y/(1+y)$. At a central zero the integral is
infinite and the eventual upper bound for $|L(1/2)|$ is immediate.
Otherwise all endpoint assertions follow by a limit from the right.
Integration of \eqref{eq:hecke-hadamard-logder} therefore yields
\begin{equation}\label{eq:hecke-shift-to-right}
 \log|L(\sigma)|-\log|L(\sigma_0)|
 \leq\frac\Delta2\log\mathcal Q-\frac\Delta2F(\sigma_0)+O_K(\Delta).
\end{equation}

Let $\Lambda_K(\nideal)$ be the ideal von Mangoldt function: it is
$\log\Norm\p$ if $\nideal=\p^j$ with $j\geq1$, and zero otherwise.
For real $s>1/2$, apply Perron's formula to $-L'/L$ with kernel
$x^w/w^2$, initially on $\Re w>\max\{0,1-s\}$. The residue at
$w=0$ is $-(L'/L)(s)\log x-(L'/L)'(s)$. A zero $\alpha$
of multiplicity $m_\alpha$ contributes
$-m_\alpha x^{\alpha-s}/(\alpha-s)^2$. Shifting left along
contours avoiding the zeros, then dividing by $\log x$ and rearranging,
gives
\begin{align}
 -\frac{L'}L(s)
 &=\sum_{\Norm\nideal\leq x}
 \frac{\Lambda_K(\nideal)\chi_{\aideal}(\nideal)}{(\Norm\nideal)^s}
       \frac{\log(x/\Norm\nideal)}{\log x}
 +\frac{(L'/L)'(s)}{\log x}\notag\\
 &\quad+\frac1{\log x}\sum_\rho\frac{x^{\rho-s}}{(\rho-s)^2}
             +\mathcal E_K(s,x).
 \label{eq:hecke-selberg-identity}
\end{align}
There is no residue from a pole of $L$, since the primitive character
is nonprincipal. The gamma factors specified above give the trivial
zero at $-m$, $m\geq0$, multiplicity
$b_m=r_2+r_1\one_{2\mid m}$. This includes the zero at $0$:
the functional equation and $L(1)\ne0$ show that the completed
function is nonzero there, so cancellation of the gamma pole is exact.
Consequently
\begin{align*}
 \mathcal E_K(u,x)&=\frac1{\log x}
       \sum_{m\geq0}b_m\frac{x^{-u-m}}{(u+m)^2},\\
 \int_{\sigma_0}^{\infty}|\mathcal E_K(u,x)|\,du
 &\ll_K\frac{x^{-\sigma_0}}{(\log x)^2}.
\end{align*}
Here $\sigma_0\geq1/2$ and $x\geq e^4$, so the bound is uniform
even as $\sigma\downarrow1/2$. The residue formula follows first
with finite contours. The usual logarithmic-derivative bounds away
from zeros justify their limit, as in \cite[Section~2]{SoundMoments}.

Take real parts and integrate from $\sigma_0$ to infinity.
The derivative term integrates to $-\Re(L'/L)(\sigma_0)/\log x$,
since $L'/L(u)\to0$ as $u\to\infty$. The nontrivial zero sum
can be integrated termwise: $\sum_\rho|\sigma_0-\rho|^{-2}<\infty$
for $\sigma_0>1/2$, and $|u-\rho|\geq|\sigma_0-\rho|$ for
$u\geq\sigma_0$. Its absolute contribution is at most
\[
 \frac1{\log x}\sum_\rho\int_{\sigma_0}^\infty
 \frac{x^{1/2-u}}{|\sigma_0-\rho|^2}\,du
 =\frac{x^{1/2-\sigma_0}}{(\sigma_0-1/2)(\log x)^2}F(\sigma_0).
\]
Using \eqref{eq:hecke-hadamard-logder} only at $\sigma_0\leq1$
(not over the entire unbounded interval of integration), we obtain
\begin{align}
 \log|L(\sigma_0)|
 &\leq\Re\sum_{\Norm\nideal\leq x}
 \frac{\Lambda_K(\nideal)\chi_{\aideal}(\nideal)}
 {(\Norm\nideal)^{\sigma_0}\log\Norm\nideal}
 \frac{\log(x/\Norm\nideal)}{\log x}
 +\frac{\log\mathcal Q}{2\log x}\notag\\
 &\quad+\left\{
 \frac{x^{1/2-\sigma_0}}{(\sigma_0-1/2)(\log x)^2}
 -\frac1{\log x}\right\}F(\sigma_0)+O_K(1).
 \label{eq:hecke-sound-at-sigma0}
\end{align}
Add \eqref{eq:hecke-shift-to-right}.  Writing
$y=(\sigma_0-1/2)\log x\geq1$, the positive part of the resulting
$F(\sigma_0)$-coefficient is
$e^{-y}/(y\log x)\leq e^{-1}/\log x$, while the two negative
contributions total $-3/(2\log x)$.  Hence the coefficient of
$F(\sigma_0)$ is at most $(e^{-1}-3/2)/\log x<0$.  This is exactly
the $\lambda=1$ instance of the sign mechanism in Soundararajan's
majorant, now with a shift starting at any $\sigma\in[1/2,3/4]$.
The conductor contributions are
$\Delta\log\mathcal Q/2+\log\mathcal Q/(2\log x)
=\log\mathcal Q/\log x$.
Removing the nonpositive zero term gives
\begin{align}
 \log|L(\sigma)|
 &\leq\Re\sum_{\Norm\nideal\leq x}
 \frac{\Lambda_K(\nideal)\chi_{\aideal}(\nideal)}
 {(\Norm\nideal)^{\sigma+1/\log x}\log\Norm\nideal}
 \frac{\log(x/\Norm\nideal)}{\log x}
 +\frac{\log\mathcal Q}{\log x}+O_K(1).
 \label{eq:grh-log-majorant-precise}
\end{align}
At a central zero the claimed upper bound is immediate. Elsewhere the
endpoint $\sigma=1/2$ follows by a limit.

Take $x=(\log X)^2$. The prime ideal theorem and partial summation give
\[
 \sum_{(\Norm\p)^j\leq x}\frac1{j(\Norm\p)^{j/2}}
 \ll_K\frac{\sqrt x}{\log x}+\log_2x
 \ll_K\frac{\log X}{\log_2X}.
\]
Since $\mathcal Q\asymp_{K,\Om}X$, the conductor term has the same bound.
For every fixed $A_0$, once $\log_2X\geq4A_0$ the required
interval $\sigma=1/2+A/\log_2X$, $0\leq A\leq A_0$, lies in
$[1/2,3/4]$. Every estimate above is uniform on that entire interval.
Only the threshold for $X$ depends on $A_0$, while the displayed
constant depends on $K,\Om$. Exponentiation finishes the proof.
\end{proof}

For a fixed point to the right of $1/2$, the ordinary short Euler product
is more convenient than the approximate functional equation.

\begin{lemma}\label{lem:short-log}
Assume \textup{GRH}$_K$ and fix $1/2<\sigma\leq1$.
For $y=(\log X)^B$ with $B$ sufficiently large in terms of $\sigma$,
the following estimates hold uniformly on $\Om(X)$.
\begin{enumerate}
\item If $\sigma<1$, then
\begin{equation}\label{eq:short-log}
 \log L(\sigma,\chi_{\aideal})=
 \sum_{\Norm\p\leq y}\frac{\chi_{\aideal}(\p)}{(\Norm\p)^\sigma}
                      +O_{K,\sigma}(1).
\end{equation}
\item At $\sigma=1$,
\begin{equation}\label{eq:short-euler-one}
 L(1,\chi_{\aideal})=
 \prod_{\Norm\p\leq y}\left(1-\frac{\chi_{\aideal}(\p)}{\Norm\p}\right)^{-1}
                 \{1+O_{K,\Om}((\log X)^{-10})\}.
\end{equation}
\end{enumerate}
\end{lemma}

\begin{proof}
GRH and the explicit formula for a nonprincipal finite-order Hecke character
give
\[
 \Psi_\chi(t):=\sum_{\Norm\nideal\leq t}
          \Lambda_K(\nideal)\chi(\nideal)
 \ll_K t^{1/2}\log^2\bigl(\mathcal Q(t+2)^{d_K}\bigr);
\]
see \cite[Chapter~5]{IK}. Partial summation therefore makes the series
for $\log L(\sigma,\chi)$ converge for $\sigma>1/2$, and bounds its tail by
\begin{align*}
 \left|\int_y^\infty\frac{t^{-\sigma}}{\log t}\,d\Psi_\chi(t)\right|
 &\ll_{K,\sigma}
 y^{1/2-\sigma}\frac{\log^2(\mathcal Q(y+2)^{d_K})}{\log y}.
\end{align*}
The identity with $\log L$ first holds for $\sigma>1$ and extends throughout
$\sigma>1/2$ by analytic continuation. Its branch is real on the real axis:
$L$ is positive for $\sigma>1$ and has no zeros for $\sigma>1/2$ under GRH.
Choosing $B$ large makes this tail $O((\log X)^{-10})$.
The terms with prime-power exponent at least two have bounded total
absolute value when $\sigma>1/2$, proving \eqref{eq:short-log}.

At $\sigma=1$ retain these powers. Completing the prime-power sums for
$\Norm\p\leq y$ changes the logarithm by $O_K(y^{-1/3})$.
Indeed primes with norm exceeding $\sqrt y$ contribute
$O_K(\sum_{\Norm\p>\sqrt y}(\Norm\p)^{-2})=O_K(y^{-1/2})$.
For primes of norm at most $\sqrt y$, each geometric tail is at most $2/y$,
and there are $O_K(\sqrt y)$ such primes. Enlarging $B$ once more and
exponentiating proves \eqref{eq:short-euler-one}.
\end{proof}

\section{G\'al sums over squarefree ideals}\label{sec:gal}

The resonator will contain many ideals, but its prime divisors must remain
small. We also need control of the products of four resonator ideals.
We obtain both properties by choosing only a few primes from each of
several disjoint blocks.

For integral ideals $\mideal,\nideal$, write $(\mideal,\nideal)$ for their
ideal gcd and $[\mideal,\nideal]$ for their ideal lcm. For a finite set
$\mathcal A$ of squarefree ideals and $s>0$, put
\[
 S_s(\mathcal A):=\sum_{\mideal,\nideal\in\mathcal A}
 \left(\frac{\Norm(\mideal,\nideal)}
 {\Norm[\mideal,\nideal]}\right)^s.
\]
The quotient in this kernel is the reciprocal of the norm of the
symmetric difference of the two prime supports. Define also
\begin{equation}\label{eq:square-energy-definition}
 E_\square(\mathcal A):=
 \#\{(\aideal_1,\aideal_2,\aideal_3,\aideal_4)\in\mathcal A^4:
 \aideal_1\aideal_2\aideal_3\aideal_4\text{ is a square}\}.
\end{equation}
This counts the terms that survive in the fourth moment of the resonator.
If squarefree products are identified with their support vectors,
$E_\square$ is additive energy in a vector space over $\mathbb F_2$.
We use the same definition for families of subsets of a finite set.

\begin{lemma}\label{lem:hamming-energy}
Let $P\geq1$ and $0\leq J\leq P/2$ be integers. Let
$\mathcal A(P,J)$ be the family of subsets of a $P$-element set having
at most $J$ elements.
\begin{enumerate}
\item We have
\begin{equation}\label{eq:hamming-energy}
 E_\square(\mathcal A(P,J))
 \leq |\mathcal A(P,J)|^2(J+1)^2
 \exp\left(4J+\frac{4J^2}{P}\right).
\end{equation}
\item If the underlying sets of the factors are disjoint, then
\begin{equation}\label{eq:product-energy}
 E_\square\left(\prod_{k=1}^t\mathcal A(P_k,J_k)\right)
 =\prod_{k=1}^t E_\square(\mathcal A(P_k,J_k)).
\end{equation}
\end{enumerate}
\end{lemma}

\begin{proof}
For the count within one layer, let $S_j$ be the family of
$j$-element subsets. If $D$ has $2r$ elements, the number of ordered pairs
$(A,B)\in S_j^2$ with $A\mathbin\triangle B=D$ is
\[
 \binom{2r}{r}\binom{P-2r}{j-r}.
\]
Indeed, $r$ elements of $D$ belong to $A\setminus B$. The rest belong to
$B\setminus A$. Their common intersection has $j-r$ elements outside $D$.
For odd $|D|$, or $|D|>2j$, the count is zero. Pairing equal symmetric
differences therefore gives
\begin{equation}\label{eq:sphere-energy-exact}
 E_\square(S_j)=\sum_{r=0}^j
 \binom P{2r}\binom{2r}r^2\binom{P-2r}{j-r}^2.
\end{equation}
For $1\leq j\leq P/2$, use
\[
 \binom P{2r}\leq\frac{P^{2r}}{(2r)!},\qquad
 \binom{2r}r\leq4^r,\qquad
 \binom{P-2r}{j-r}\leq\frac{P^{j-r}}{(j-r)!},\qquad
 \binom Pj\geq\frac{(P-j)^j}{j!}.
\]
Since $-\log(1-j/P)\leq2j/P$, these inequalities imply
\begin{align*}
 \frac{E_\square(S_j)}{|S_j|^2}
 &\leq\left(\frac P{P-j}\right)^{2j}
 \sum_{r=0}^j\frac{16^r(j!/(j-r)!)^2}{(2r)!}\\
 &\leq e^{4j^2/P}\sum_{r=0}^j\frac{(4j)^{2r}}{(2r)!}
 \leq e^{4j+4j^2/P}.
\end{align*}
The same bound holds for $j=0$.

The ball also contains quadruples from different layers. To retain these
terms, work in $G=\mathbb F_2^P$ and define
$\widehat f(\xi)=\sum_{x\in G}f(x)(-1)^{x\cdot\xi}$.
For subsets $B_1,\ldots,B_4$ of $G$, Fourier inversion and H\"older's
inequality give
\begin{align*}
 \#\{(x_i)\in\textstyle\prod_i B_i:\sum_i x_i=0\}
 &=|G|^{-1}\sum_{\xi\in G}\prod_{i=1}^4
 \widehat{\one_{B_i}}(\xi)\\
 &\leq\prod_{i=1}^4 E_\square(B_i)^{1/4}.
\end{align*}
Here $|G|^{-1}\sum_\xi|\widehat{\one_B}(\xi)|^4=E_\square(B)$.
Expand the ball into its disjoint layers and apply this inequality. Then
\begin{align*}
 E_\square(\mathcal A(P,J))
 &\leq\left(\sum_{j=0}^J E_\square(S_j)^{1/4}\right)^4\\
 &\leq e^{4J+4J^2/P}
 \left(\sum_{j=0}^J |S_j|^{1/2}\right)^4\\
 &\leq e^{4J+4J^2/P}(J+1)^2
 \left(\sum_{j=0}^J|S_j|\right)^2.
\end{align*}
This proves the first assertion. In a product of disjoint blocks, the
square condition holds in each block separately. Its count therefore
factorizes, proving the second assertion.
\end{proof}

We next estimate the gain from one block. Two products in the block may
share most of their primes. The kernel then charges only the primes
that differ. The following calculation balances the number of such
pairs against their weights.

\begin{lemma}\label{lem:squarefree-block}
Let $\mathscr P$ have $P$ elements, with weights $0<w_p<1$. Put
\[
 H:=\sum_{p\in\mathscr P}w_p,\qquad W:=\max_{p\in\mathscr P}w_p,
 \qquad
 S_{\mathscr P}(J):=\sum_{A,B\in\mathcal A(P,J)}
 \prod_{p\in A\triangle B}w_p.
\]
Suppose
\[
 1\le r\le J/2,\qquad J\le P/3,\qquad
 (J+2r)W/H\le1/2.
\]
Then, with an absolute implied constant,
\begin{align}
 \log\frac{S_{\mathscr P}(J)}{|\mathcal A(P,J)|}
 &\geq r\log\left(\frac{e^2JH^2}{Pr^2}\right)\notag\\
 &\quad-O\!\left(
 r\frac{(J+2r)W}{H}+\frac{r^2}{J}+\frac{rJ}{P}
 +\log(r+1)\right).
 \label{eq:squarefree-block-quantitative}
\end{align}
Consequently, if
\[
 r\longrightarrow\infty,\qquad r/J\longrightarrow0,\qquad
 J/P\longrightarrow0,\qquad (J+r^2)W/H\longrightarrow0,
\]
then, uniformly under these hypotheses,
\begin{equation}\label{eq:squarefree-block}
 \log\frac{S_{\mathscr P}(J)}{|\mathcal A(P,J)|}
 \geq r\log\left(\frac{e^2JH^2}{Pr^2}\right)+o(r).
\end{equation}
\end{lemma}

\begin{proof}
Restrict the sum to pairs
$A=D\mathbin{\dot\cup}X$ and $B=D\mathbin{\dot\cup}Y$, where
$|D|=J-r$, $|X|=|Y|=r$, and $D,X,Y$ are pairwise disjoint.
We need a lower bound for the weighted choices of $X$ and $Y$ after
some primes have been excluded.

Let $E$ be an excluded set with $|E|\leq J+r$. After $j$ distinct entries
have been chosen, the total weight available for the next entry is at
least $H-(|E|+j)W$. Thus the total weight of ordered distinct $r$-tuples
outside $E$ lies between
\[
 \prod_{j=0}^{r-1}\bigl(H-(|E|+j)W\bigr)
 \quad\text{and}\quad H^r.
\]
By the hypothesis $(J+2r)W/H\le1/2$, every factor in this product is
positive.  Using $\log(1-x)\ge-2x$ for $0\le x\le1/2$ gives the
quantitative lower bound
\[
 \prod_{j=0}^{r-1}\bigl(H-(|E|+j)W\bigr)
 \ge H^r\exp\left(-O\left(\frac{r(J+2r)W}{H}\right)\right).
\]
Dividing by $r!$ gives the corresponding sum over unordered subsets.
Apply this first with $E=D$, then with $E=D\cup X$. Since $D$ has
$\binom P{J-r}$ possible values,
\[
 S_{\mathscr P}(J)\geq
 \binom P{J-r}\frac{H^{2r}}{(r!)^2}
 \exp\left(-O\left(\frac{r(J+2r)W}{H}\right)\right).
\]
Since $J\le P/3$ and $r\le J/2$, comparison of successive
binomial coefficients and a logarithmic expansion give
\[
 |\mathcal A(P,J)|\leq2\binom PJ,
 \qquad
 \frac{\binom P{J-r}}{\binom PJ}
 \geq\left(\frac JP\right)^r
 \exp\left(-O\left(\frac{r^2}{J}+\frac{rJ}{P}\right)\right).
\]
For the second bound, write the ratio as
$\prod_{j=0}^{r-1}(J-j)/(P-J+j+1)$ and take logarithms.
Stirling's formula now gives
\begin{align*}
 \log\frac{S_{\mathscr P}(J)}{|\mathcal A(P,J)|}
 &\geq r\log\frac JP+2r\log H-2r\log r+2r\\
 &\quad-O\left(\frac{r(J+2r)W}{H}+\frac{r^2}{J}
 +\frac{rJ}{P}+\log(r+1)\right).
\end{align*}
This proves \eqref{eq:squarefree-block-quantitative}.  Under the asymptotic
hypotheses of the lemma, every term in the error is $o(r)$, uniformly,
and \eqref{eq:squarefree-block} follows.
\end{proof}

\begin{proposition}\label{prop:gal}Fix $S$. For every sufficiently large integer $N$, there is a set
$\mathcal M$ of $N$ squarefree integral ideals coprime to $S$ with the
following properties.
\begin{enumerate}
\item Its G\'al sum satisfies
\begin{equation}\label{eq:gal}
 S_{1/2}(\mathcal M)\geq
 N\exp\left\{(2+o(1))
 \sqrt{\frac{\log N\log_3N}{\log_2N}}\right\}.
\end{equation}
More generally, uniformly for $A$ in any fixed compact subset of
$[0,\infty)$,
\begin{equation}\label{eq:moving-gal}
 S_{1/2+A/\log_2N}(\mathcal M)\geq
 N\exp\left\{(2e^{-A}+o(1))
 \sqrt{\frac{\log N\log_3N}{\log_2N}}\right\}.
\end{equation}
\item Its prime support and square energy satisfy
\begin{equation}\label{eq:yM}
 y_{\mathcal M}:=
 \max_{\substack{\mideal\in\mathcal M\\\p\mid\mideal}}\Norm\p
 \leq(\log N)^{1+o(1)},
\end{equation}
and
\begin{equation}\label{eq:gal-energy}
 E_\square(\mathcal M)\leq N^{2+o(1)}.
\end{equation}
\end{enumerate}
\end{proposition}

\begin{proof}
The construction uses blocks of geometrically increasing prime norms.
The number of primes selected from the $k$th block decreases like
$k^{-2}$. This makes the logarithm of the cardinality and the gain in
the kernel depend on the same harmonic sum.

First fix $1<u\leq e$, $0<\gamma<1$, and $0<\eta<1/2$.
Write $T_i=\log_iN$ for $1\leq i\leq4$, and put
$K_N=\lfloor T_2^\gamma\rfloor$.
Let $\mathscr P_k$ consist of the prime ideals outside $S$ whose norms
belong to
\[
 (u^kT_1T_2,u^{k+1}T_1T_2],\qquad 1\leq k\leq K_N.
\]
Set $P_k=|\mathscr P_k|$ and
\[
 H_k:=\sum_{\p\in\mathscr P_k}(\Norm\p)^{-1/2},
 \qquad W_k:=\max_{\p\in\mathscr P_k}(\Norm\p)^{-1/2}.
\]
For fixed $u$ and $\gamma$, the block endpoints satisfy
$\log(u^kT_1T_2)=T_2+o(T_2)$ uniformly for $k\le K_N$.  Thus the
prime ideal theorem (and partial summation) may be applied uniformly over
all the blocks, giving
\begin{align}
 P_k&=u^k(u-1)T_1\{1+o(1)\},\label{eq:nf-block-P}\\
 H_k&=2u^{k/2}(\sqrt u-1)\sqrt{T_1/T_2}\{1+o(1)\}.
 \label{eq:nf-block-H}
\end{align}
Here the logarithm of every norm in a block is $T_2+o(T_2)$.
In particular, with
$\lambda(u)=4(\sqrt u-1)/(\sqrt u+1)$, we have
\begin{equation}\label{eq:nf-block-density}
 \frac{H_k^2}{P_k}=\frac{\lambda(u)+o(1)}{T_2},
 \qquad \frac{W_k}{H_k}\ll_u\frac{u^{-k}}{T_1}.
\end{equation}

For a fixed $a>0$, put
$J_k(a)=\lfloor aT_1/(k^2T_3)\rfloor$.
Let $\mathcal M(a)$ contain all products obtained by choosing at most
$J_k(a)$ primes from each $\mathscr P_k$. We claim that
\begin{equation}\label{eq:nf-entropy}
 \log|\mathcal M(a)|=(a\gamma\log u+o(1))\log N.
\end{equation}
To verify the cardinality, first note the uniform estimate
\begin{equation}\label{eq:binomial-tail-entropy}
 \log\sum_{\nu=0}^J\binom P\nu
 =J\log\frac PJ+J+
 O\left(\log(J+1)+\frac{J^2}{P}\right)
 \qquad(1\leq J,\ J/P=o(1)).
\end{equation}
Indeed,
$\binom P{J-\ell}/\binom PJ\leq(J/(P-J+1))^\ell$,
so the sum is $\binom PJ(1+O(J/P))$.
Apply Stirling's formula to
\[
 \log\binom PJ=J\log P-\log J!+
 \sum_{\nu=0}^{J-1}\log(1-\nu/P).
\]
The last sum is $O(J^2/P)$, proving the estimate.

For our quotas,
\[
 \log\frac{P_k}{J_k(a)}
 =k\log u+2\log k+T_4+O_{u,a}(1).
\]
The first term provides the main contribution, since
\[
 \sum_{k\leq K_N}kJ_k(a)
 =\frac{aT_1}{T_3}\sum_{k\leq K_N}\frac1k+O(K_N^2)
 =(a\gamma+o(1))T_1.
\]
The remaining terms satisfy
\begin{align*}
 \sum_{k\leq K_N}J_k(a)(1+\log k)&=O_{u,a}(T_1/T_3)+o(T_1),\\
 T_4\sum_{k\leq K_N}J_k(a)&=O_{u,a}(T_1T_4/T_3)+o(T_1)=o(T_1),\\
 \sum_{k\leq K_N}\log(J_k(a)+1)&\leq K_NT_2=o(T_1),\\
 \sum_{k\leq K_N}\frac{J_k(a)^2}{P_k}
 &\ll_{u,a}\frac{T_1}{T_3^2}
 \sum_{k\geq1}\frac{u^{-k}}{k^4}=o(T_1).
\end{align*}
Also $K_N^2=o(T_1)$, so the floor errors are negligible.
Summing \eqref{eq:binomial-tail-entropy} proves \eqref{eq:nf-entropy}.

To bring the cardinality close to $N$, we adjust the quotas.
Put $a_\pm=(1\pm\eta)/(\gamma\log u)$.
For large $N$, $|\mathcal M(a_-)|<N<|\mathcal M(a_+)|$.
Starting with the lower quotas, increase one quota by one at each step,
without exceeding the upper quotas. Stop before the cardinality first
exceeds $N$. Denote the resulting family by $\mathcal M_*$ and its
cardinality by $M$. Its quotas satisfy
$J_k(a_-)\leq J_k\leq J_k(a_+)$, and
\begin{equation}\label{eq:nf-near-N}
 M\leq N,\qquad N/M\leq1+\max_{k\leq K_N}P_k
 \leq(\log N)^{1+o(1)}.
\end{equation}
For the middle inequality, a single quota increase changes the
cardinality by at most
\[
 \frac{\sum_{j\leq J+1}\binom Pj}{\sum_{j\leq J}\binom Pj}
 \leq1+\frac{\binom P{J+1}}{\binom PJ}\leq1+P.
\]

We estimate the central and shifted kernels together. Fix $A_0\geq0$
and let $0\leq A\leq A_0$. Define
\[
 H_k(A):=\sum_{\p\in\mathscr P_k}
 (\Norm\p)^{-1/2-A/T_2},\qquad
 W_k(A):=\max_{\p\in\mathscr P_k}(\Norm\p)^{-1/2-A/T_2}.
\]
Since $\log\Norm\p=T_2+o(T_2)$ uniformly on the blocks,
\begin{equation}\label{eq:moving-Hk}
 H_k(A)=e^{-A}H_k\{1+o(1)\},\qquad
 W_k(A)/H_k(A)=\{1+o(1)\}W_k/H_k.
\end{equation}
Thus a shift of $A/T_2$ multiplies the available weight in each block
by $e^{-A}+o(1)$. Choose
\[
 \alpha=\sqrt{a_-\lambda(u)},\qquad
 r_k(A)=\left\lfloor\frac{e^{-A}\alpha}{k}
 \sqrt{\frac{T_1}{T_2T_3}}\right\rfloor.
\]
These values balance the gain from the weights with the cost of
choosing different primes. We check the hypotheses of the preceding
lemma, uniformly for $k\leq K_N$ and $0\leq A\leq A_0$:
\begin{align*}
 \min_k r_k(A)&\gg_{u,\gamma,\eta,A_0}
 \frac{\sqrt{T_1}}{T_2^{\gamma+1/2}\sqrt{T_3}}\longrightarrow\infty,\\
 \max_k\frac{r_k(A)}{J_k}&\ll_{u,\gamma,\eta,A_0}
 \frac{K_N\sqrt{T_3}}{\sqrt{T_1T_2}}\longrightarrow0,\\
 \max_k\frac{J_k}{P_k}&\ll_{u,\gamma,\eta}T_3^{-1}\longrightarrow0,\\
 \frac{(J_k+r_k(A)^2)W_k(A)}{H_k(A)}
 &\ll_{u,\gamma,\eta,A_0}
 \frac{u^{-k}}{k^2T_3}+\frac{u^{-k}}{k^2T_2T_3}=o(1).
\end{align*}
All four $o(1)$ statements are uniform in $k\le K_N$ and
$0\le A\le A_0$.  In particular, the quantitative error in
\eqref{eq:squarefree-block-quantitative}, divided by $r_k(A)$, tends to
zero uniformly.  Since $\min_k r_k(A)\to\infty$, summing the block errors
therefore contributes
\[
 o\!\left(\sum_{k\le K_N}r_k(A)\right)
\]
uniformly for $0\le A\le A_0$.
Moreover, \eqref{eq:nf-block-density} and the lower quota bound imply
\[
 \frac{e^2J_kH_k(A)^2}{P_kr_k(A)^2}
 \geq e^2\{1+o(1)\}.
\]
The kernel factorizes across the disjoint blocks. Apply
Lemma~\ref{lem:squarefree-block} and sum its logarithmic bounds to obtain
\begin{align}
 \log\frac{S_{1/2+A/T_2}(\mathcal M_*)}{M}
 &\geq(2+o(1))\sum_{k\leq K_N}r_k(A)\notag\\
 &=\left(2e^{-A}\sqrt{\gamma(1-\eta)
 \frac{\lambda(u)}{\log u}}+o(1)\right)
 \sqrt{\frac{T_1T_3}{T_2}}.
 \label{eq:nf-gal-intermediate}
\end{align}
Here $\sum_{k\leq K_N}1/k=\gamma T_3+O(1)$, and the total floor error
$O(K_N)$ is $o(\sqrt{T_1T_3/T_2})$.

To obtain exactly $N$ ideals, first double the cardinality as often as
possible using new prime divisors. Put
$h=\lfloor\log(N/M)/\log2\rfloor$.
By \eqref{eq:nf-near-N}, $h=O(T_2)$.
Choose distinct prime ideals $\mathfrak q_1,\ldots,\mathfrak q_h$ outside
$S$ and the original blocks, with norms at most $T_2^{O_K(1)}$.
The prime ideal theorem supplies these primes below the first block.
Let $\mathcal C_h$ be the set of divisors of their product, and put
$\mathcal M_1=\mathcal C_h\mathcal M_*$.
Then $N/2<|\mathcal M_1|\leq N$. For every $s>0$,
\[
 \frac{S_s(\mathcal C_h)}{|\mathcal C_h|}
 =\prod_{\ell=1}^h(1+(\Norm\mathfrak q_\ell)^{-s})\geq1,
\]
and disjointness of the supports gives
\begin{equation}\label{eq:nf-cube-gal-factorization}
 \frac{S_s(\mathcal M_1)}{|\mathcal M_1|}
 =\frac{S_s(\mathcal C_h)}{|\mathcal C_h|}
 \frac{S_s(\mathcal M_*)}{M}
 \geq\frac{S_s(\mathcal M_*)}{M}.
\end{equation}
Now add unused elements of $\mathcal M(a_+)$ until the cardinality is $N$.
There are enough: the auxiliary cube uses primes disjoint from all
original blocks, so
$\mathcal M(a_+)\cap\mathcal M_1=\mathcal M_*$.  Hence the number
of available new elements is $|\mathcal M(a_+)|-M>N-M$, whereas at
most $N-|\mathcal M_1|\le N-M$ elements are needed.
Call the completed set $\mathcal M$. Positivity of the kernel yields
\[
 \frac{S_s(\mathcal M)}N
 \geq\frac{|\mathcal M_1|}{N}
 \frac{S_s(\mathcal M_1)}{|\mathcal M_1|}
 >\frac12\frac{S_s(\mathcal M_*)}{M}.
\]
Thus completing the cardinality costs only a bounded factor.

We must also check that the completion has small square energy.
Lemma~\ref{lem:hamming-energy} gives
\begin{align}
 \log\frac{E_\square(\mathcal M(a_+))}{|\mathcal M(a_+)|^2}
 &\ll\sum_{k\leq K_N}\left(J_k(a_+)+
 \frac{J_k(a_+)^2}{P_k}+\log(J_k(a_+)+1)\right)\notag\\
 &\ll_{u,\gamma,\eta}T_1/T_3+T_1/T_3^2+K_NT_2=o(T_1).
 \label{eq:nf-energy-upper-family}
\end{align}
All quotas are at most $P_k/2$ for large $N$, as the lemma requires.
Every element of $\mathcal M$ belongs to
$\mathcal U=\mathcal C_h\mathcal M(a_+)$.
The full cube $\mathcal C_h$ has square energy $2^{3h}$: three entries
determine the fourth. Therefore
\[
 E_\square(\mathcal M)\leq E_\square(\mathcal U)
 =2^{3h}E_\square(\mathcal M(a_+))
 \leq N^{2+2\eta+o(1)}.
\]
We used $2^h\leq N/M\leq(\log N)^{1+o(1)}$ and
$|\mathcal M(a_+)|=N^{1+\eta+o(1)}$.

Finally, let the fixed parameters approach their limiting values by
diagonal selection. For each integer $j\geq3$, take
\[
 u_j=e^{1/j},\qquad \gamma_j=1-1/j,\qquad \eta_j=1/j.
\]
All estimates above hold for these fixed parameters, uniformly for
$0\leq A\leq j$. More explicitly, for each fixed $j$ every $o(1)$
occurring in the block estimates may be chosen so that its supremum over
$0\leq A\leq j$ and $1\leq k\leq K_N$ tends to zero as $N\to\infty$.
Thus one may choose a single threshold $N_j$ which works simultaneously
for all these parameters. Choose increasing thresholds $N_j$ so that, for
$N\geq N_j$, every normalized error is at most $1/j$, and
$E_\square(\mathcal U)\leq N^{2+5/j}$.
We may also require
\[
 \frac{\log9}{\log N}\left(h+\sum_{k\leq K_N}J_k(a_+)\right)
 \leq\frac1j;
\]
for each fixed $j$, the sum in parentheses is
$O_j(\log N/\log_3N+\log_2N)=o(\log N)$.
This additional condition records the degree bound used in
Remark~\ref{rem:hypercontractive-energy}.
On $N_j\leq N<N_{j+1}$ use the $j$th construction.
Since $\lambda(u)/\log u\to1$ as $u\downarrow1$, the coefficient in
\eqref{eq:nf-gal-intermediate} tends to $2e^{-A}$, uniformly on every
fixed compact $A$-interval. This proves \eqref{eq:moving-gal} and its
case $A=0$, namely \eqref{eq:gal}. The energy bound follows from
$j=j(N)\to\infty$.

The largest norm in an original block is at most
$u_j^{K_N+1}T_1T_2$. Its logarithm satisfies
\[
 \frac{\log(u_j^{K_N+1}T_1T_2)}{T_2}
 =1+\frac{T_3}{T_2}
 +\frac{K_N+1}{jT_2}
 \leq1+o(1)+\frac1j.
\]
The auxiliary primes have smaller norms, proving \eqref{eq:yM}.
All ideals are squarefree and coprime to $S$ by construction.
\end{proof}

\begin{remark}[Energy and hypercontractivity]
\label{rem:hypercontractive-energy}
For a finite nonempty set $\mathcal A$ of squarefree ideals, put
\[
 \Delta(\mathcal A)=\max_{\mideal\in\mathcal A}
                         \#\{\p:\p\mid\mideal\}.
\]
Let $x_\p$ be independent uniform signs indexed by the union of its
prime supports, and define
\[
 F_{\mathcal A}(x)=\sum_{\mideal\in\mathcal A}
                         \prod_{\p\mid\mideal}x_\p.
\]
This multilinear polynomial has degree at most $\Delta(\mathcal A)$.
Orthogonality of the sign monomials gives
\[
 \mathbb E F_{\mathcal A}^2=|\mathcal A|,\qquad
 \mathbb E F_{\mathcal A}^4=E_\square(\mathcal A).
\]
Bonami hypercontractivity
\cite[Theorem~9.21, with $q=4$]{ODonnell} states that
$\|F_{\mathcal A}\|_4\leq
3^{\Delta(\mathcal A)/2}\|F_{\mathcal A}\|_2$.
Consequently,
\begin{equation}\label{eq:bonami-square-energy}
 E_\square(\mathcal A)\leq
                 9^{\Delta(\mathcal A)}|\mathcal A|^2.
\end{equation}
For the completed set in Proposition~\ref{prop:gal},
\[
 \Delta(\mathcal M)\leq h+\sum_{k\leq K_N}J_k(a_+)
 \leq h+\frac{\pi^2a_+}{6}\frac{\log N}{\log_3N}.
\]
The thresholds chosen above ensure $\Delta(\mathcal M)=o(\log N)$,
so \eqref{eq:bonami-square-energy} also proves
$E_\square(\mathcal M)\leq N^{2+o(1)}$ directly on the completed set.
Thus the energy estimate is a consequence of sparsity. The separate
block and truncation estimates establish that this sparsity is compatible
with the required G\'al gain, prime support, and approximate functional
equation. Lemma~\ref{lem:hamming-energy} provides a self-contained
combinatorial route to the energy bound used in the proof.
\end{remark}

The approximate functional equation only sees pairs whose symmetric
difference has bounded norm. The next lemma shows that these pairs
already carry the required G\'al mass.

\begin{lemma}\label{lem:gal-truncation}Fix $\eps>0$, and let $\mathcal M$ be supplied by
Proposition~\ref{prop:gal}.
\begin{enumerate}
\item If $N\leq X^{1/4}$, then
\[
 \sum_{\substack{\mideal,\nideal\in\mathcal M\\
 \Norm[\mideal,\nideal]/\Norm(\mideal,\nideal)\leq X^\eps}}
 \sqrt{\frac{\Norm(\mideal,\nideal)}{\Norm[\mideal,\nideal]}}
 \geq N\exp\left\{(2+o(1))
 \sqrt{\frac{\log N\log_3N}{\log_2N}}\right\}.
\]
\item Suppose $N=X^{\beta+o(1)}$, where $0<\beta\leq1/4$ is fixed.
Uniformly for $A$ in any fixed compact subset of $[0,\infty)$,
\begin{align*}
 &\sum_{\substack{\mideal,\nideal\in\mathcal M\\
 \Norm[\mideal,\nideal]/\Norm(\mideal,\nideal)\leq X^\eps}}
 \left(\frac{\Norm(\mideal,\nideal)}{\Norm[\mideal,\nideal]}\right)^{1/2+A/\log_2X}\\
 &\qquad\geq N\exp\left\{(2e^{-A}+o(1))
 \sqrt{\frac{\log N\log_3N}{\log_2N}}\right\}.
\end{align*}
\end{enumerate}
\end{lemma}

\begin{proof}
For an omitted pair put
$d=\Norm[\mideal,\nideal]/\Norm(\mideal,\nideal)>X^\eps$.
Then $d^{-1/2}\leq X^{-\eps/6}d^{-1/3}$.
Fix $\mideal$ and enlarge the sum over $\nideal$ to all squarefree
products of primes with norms at most $y_{\mathcal M}$.
At each prime, the two choices contribute $1+(\Norm\p)^{-1/3}$,
regardless of whether $\p$ divides $\mideal$. Hence
\[
 \sum_{\nideal\in\mathcal M}d^{-1/3}
 \leq\prod_{\Norm\p\leq y_{\mathcal M}}
 (1+(\Norm\p)^{-1/3})
 \leq\exp\{O_K(y_{\mathcal M}^{2/3})\}.
\]
By \eqref{eq:yM}, the total omitted central mass is at most
\[
 NX^{-\eps/6}\exp\{(\log N)^{2/3+o(1)}\}=o(N)
\]
in either range for $N$. Indeed $N\leq X^{1/4+o(1)}$ in the applications
here, so $(\log N)^{2/3+o(1)}=o(\log X)$ and the exponential factor
cannot offset the fixed power $X^{-\eps/6}$. The omitted shifted mass is
no larger.
Subtracting this bound from \eqref{eq:gal} proves the first assertion.

For the second, put $B_X=A\log_2N/\log_2X$.
Then $B_X=A+o(1)$ uniformly for bounded $A$, and
$B_X/\log_2N=A/\log_2X$ exactly.
Apply the compact-uniform estimate \eqref{eq:moving-gal} with $B_X$.
Since $e^{-B_X}=e^{-A}+o(1)$, its main term is the one asserted above.
Subtract the same $o(N)$ tail to finish the proof.
\end{proof}

Fix one set $\mathcal M$ from the proposition for every sufficiently
large $N$. The two bounds in its second assertion can be recorded with
one function. Define
\begin{equation}\label{eq:common-energy-omega}
 \omega(t):=\sup_{N\geq t}
 \max\left\{0,
 \frac{\log y_{\mathcal M}}{\log_2N}-1,
 \frac{\log E_\square(\mathcal M)}{\log N}-2\right\}.
\end{equation}
Then $\omega(t)\to0$, and the chosen sets satisfy
$y_{\mathcal M}\leq(\log N)^{1+\omega(N)}$ and
$E_\square(\mathcal M)\leq N^{2+\omega(N)}$.
These are the uniform bounds used in the fourth-moment argument.

\section{Large values near the central point}
\label{sec:central}We prove all three assertions of Theorem~\ref{thm:central-max}
together.  The same argument covers the central point by taking $A=0$.
The resonator assigns extra weight to characters whose values agree on
many small prime ideals.  The G\'al sum measures the resulting gain in
the first moment.  A fourth-moment estimate then shows that this gain
cannot come from too few characters.

Fix a nonzero function $W\in C_c^\infty((1,2))$ with $0\leq W\leq1$,
and put
\[
 I_W:=\int_0^\infty W(t)\,dt>0.
\]
For a finite set $\mathcal M$ of ideals, write
\[
 R_{\aideal}:=\sum_{\mideal\in\mathcal M}
 \chi_{\aideal}(\mideal).
\]
These sums are real because the characters are quadratic.

\begin{lemma}\label{lem:nf-fourth-moment}
Assume \textup{GRH}$_K$.  Fix $0<\beta<1/4$, and put
$N=\lfloor X^\beta\rfloor$.
Let $\mathcal M$ be an $N$-element set of squarefree ideals coprime to $S$.
Suppose that a fixed function $\omega(t)\to0$ satisfies
\[
 y_{\mathcal M}\leq(\log N)^{1+\omega(N)},
 \qquad E_\square(\mathcal M)\leq N^{2+\omega(N)}.
\]
Then
\begin{equation}\label{eq:nf-fourth-moment}
 \sum_{\aideal\in\Om}W(\Norm\aideal/X)|R_{\aideal}|^4
 \ll_{K,\Om,W,\beta}XN^{2+o(1)}.
\end{equation}
\end{lemma}

\begin{proof}
The fourth power counts quadruples of resonator ideals.
The character average distinguishes those whose product is a square.
Since $|R_{\aideal}|^4=R_{\aideal}^4$,
Proposition~\ref{prop:orthogonality} gives
\begin{align}
 &\sum_{\aideal\in\Om}W(\Norm\aideal/X)|R_{\aideal}|^4\notag\\
 &\quad=\kappa_{\Om}I_WX
 \sum_{\substack{\mideal_1,\ldots,\mideal_4\in\mathcal M\\
 \mideal_1\mideal_2\mideal_3\mideal_4=\square}}
 h_K(\mideal_1\mideal_2\mideal_3\mideal_4)
 +O\left(X^{1/2+\eps+o(1)}N^4\right).
 \label{eq:nf-fourth-expanded}
\end{align}
Here Lemma~\ref{lem:FG-uniform} applies with $\uideal=(1)$.
Indeed, for every fixed $\eps>0$, the support hypothesis gives
$y_{\mathcal M}\leq(\log X)^{1+\eps/4}$ once $X$ is large.
The integral main term is at most
$\kappa_{\Om}I_WX E_\square(\mathcal M)$, since $0<h_K\leq1$.
Choose $0<\eps<1/2-2\beta$.
The error divided by $XN^2$ is
\[
 O\left(X^{-1/2+\eps+o(1)}N^2\right)
 =O\left(X^{-1/2+\eps+2\beta+o(1)}\right)=o(1).
\]
The assumed energy bound now proves the lemma.
\end{proof}

\begin{proof}[Proof of Theorem~\ref{thm:central-max}]
Fix $A_0\geq0$ and $0<\beta<1/4$.
Throughout the proof, $0\leq A\leq A_0$ and
\[
 N=\lfloor X^\beta\rfloor,\qquad
 \delta=\frac{A}{\log_2X},\qquad s=\frac12+\delta.
\]
Take $\mathcal M$ from Proposition~\ref{prop:gal}, with the
common support and energy bounds in \eqref{eq:common-energy-omega}.
Define
\[
 S_0:=\sum_{\aideal\in\Om}W(\Norm\aideal/X)R_{\aideal}^2,
 \qquad
 S_1(A):=\sum_{\aideal\in\Om}W(\Norm\aideal/X)
 L(s,\chi_{\aideal})R_{\aideal}^2.
\]
Write $S_1=S_1(0)$ at the central point.

For squarefree ideals $\mideal,\nideal$, their product is a square
if and only if $\mideal=\nideal$.
Expand $S_0$ and apply Proposition~\ref{prop:orthogonality}
and Lemma~\ref{lem:FG-uniform}.  This gives
\[
 S_0=\kappa_{\Om}I_WX
 \sum_{\mideal\in\mathcal M}h_K(\mideal^2)
 +O\left(X^{1/2+\eps+o(1)}N^2\right).
\]
We fix
$0<\eps<\min\{1/4-\beta,1/4\}$ for the rest of the moment calculation.
Since $h_K\leq1$, we obtain
\begin{equation}
 S_0\leq(\kappa_{\Om}I_W+o(1))XN.
 \label{eq:S0-central}
\end{equation}

The first moment contains a second ideal sum from the approximate
functional equation.  Its weights must retain their decay when we
apply the character-average error bound.
For $t>0$, put
\[
 \mathcal Q(t):=|D_K|\Norm\cideal_{\Om}\,t.
\]
The two test functions in \eqref{eq:shifted-afe} are
\begin{align*}
 w^+_{\kideal,X,A}(t)
 &:=W(t/X)V_{K,\delta}
       \left(\frac{\Norm\kideal}{\sqrt{\mathcal Q(t)}}\right),\\
 w^-_{\kideal,X,A}(t)
 &:=W(t/X)\mathcal Q(t)^{-\delta}
 \frac{\gamma_K(1/2-\delta)}{\gamma_K(1/2+\delta)}
 V_{K,-\delta}
       \left(\frac{\Norm\kideal}{\sqrt{\mathcal Q(t)}}\right).
\end{align*}
For large $X$, we have $0\leq\delta\leq\delta_0<1/4$, where
$\delta_0$ is fixed.  Differentiating
$\Norm\kideal/\sqrt{\mathcal Q(t)}$ multiplies its logarithmic
derivatives by fixed powers of $t^{-1}$.
Moreover,
$t^j(d/dt)^j\mathcal Q(t)^{-\delta}=O_j(X^{-\delta})$
on $[X,2X]$, uniformly in this range of $\delta$.
Thus \eqref{eq:shifted-kernel-decay} gives, for every fixed $B>0$,
\begin{align*}
 \|w^+_{\kideal,X,A}\|_{*,X}
 &\ll_{K,W,A_0,B}
       (1+\Norm\kideal/\sqrt X)^{-B},\\
 \|w^-_{\kideal,X,A}\|_{*,X}
 &\ll_{K,W,A_0,B}X^{-\delta}
       (1+\Norm\kideal/\sqrt X)^{-B}.
\end{align*}
The ideal-counting bound
$\#\{\kideal:\Norm\kideal\leq t\}\ll_Kt$
and partial summation imply
\[
 \sum_{\kideal}(\Norm\kideal)^{-a}
 (1+\Norm\kideal/\sqrt X)^{-B}
 \ll_{K,B}X^{(1-a)/2}
 \qquad(1/4\leq a\leq3/4, B\geq2).
\]
Consequently,
\begin{equation}\label{eq:critical-first-length}
 \sum_{\kideal}(\Norm\kideal)^{-1/2-\delta}
 \|w^+_{\kideal,X,A}\|_{*,X}
 \ll X^{1/4-\delta/2}\leq X^{1/4},
\end{equation}
and
\begin{equation}\label{eq:critical-dual-length}
 \sum_{\kideal}(\Norm\kideal)^{-1/2+\delta}
 \|w^-_{\kideal,X,A}\|_{*,X}
 \ll X^{-\delta}X^{1/4+\delta/2}\leq X^{1/4}.
\end{equation}
These estimates concern the seminorms themselves, so they apply
directly to the errors in Proposition~\ref{prop:orthogonality}.

We may first truncate both ideal sums at $\Norm\kideal\leq X^2$.
Indeed, the decay estimate with $B$ sufficiently large gives
\[
 \sum_{\Norm\kideal>X^2}
 (\Norm\kideal)^{-1/2+\delta}
 (1+\Norm\kideal/\sqrt X)^{-B}
 \ll X^{B/2+2(1/2+\delta-B)}.
\]
Since $\#\Om(X)\ll X$ and $|R_{\aideal}|\leq N$, the contribution
of these tails to $S_1(A)$ is $O(X^{-100})$ after increasing $B$.
For the remaining ideals, Lemma~\ref{lem:FG-uniform} applies with
$\uideal=\kideal$ and $\videal=\mideal\nideal$, taking $B_0=2$.
More explicitly, if $\kideal\mideal\nideal=\lideal_0\lideal_1^2$
with $\lideal_0$ squarefree, there is a constant $C_{K,\eps}$ such that
\[
 \sup_{\substack{\Norm\kideal\leq X^2\\
                  \mideal,\nideal\in\mathcal M}}
 F_\eps(\lideal_0)G_\eps(\lideal_1)
 \leq\exp\{C_{K,\eps}(\log X)^{1-\eps/2}\}=X^{o(1)}.
\]
This bound uses only the common prime support of $\mathcal M$, so it
also holds uniformly when $\mathcal M$ varies with $0\leq A\leq A_0$.
It can therefore be taken outside the whole truncated ideal sum and
both resonator sums before applying
\eqref{eq:critical-first-length}--\eqref{eq:critical-dual-length}.
The same bound, with $\uideal=(1)$, applies to all products of two
or four resonator ideals in $S_0$ and the fourth moment. No assertion
of uniformity over the untruncated $\kideal$-sum is needed: its tail
was bounded absolutely above before applying character orthogonality.
Terms divisible by a prime in $S$ vanish by
\eqref{eq:bad-prime-terms-vanish}.
Equations \eqref{eq:critical-first-length} and
\eqref{eq:critical-dual-length} therefore bound the combined
character-average error by
\begin{equation}\label{eq:critical-average-error}
 O_{K,\Om,W,A_0,\eps}
       \left(X^{3/4+\eps+o(1)}N^2\right)=o(XN).
\end{equation}
The last equality uses $\eps<1/4-\beta$.

Both square-diagonal main terms are nonnegative.
This follows from the positivity of the kernels, the dual factor
\eqref{eq:dual-factor}, and $h_K$.
In the first main term, retain, for each pair
$\mideal,\nideal\in\mathcal M$, only
\[
 \kideal=\frac{[\mideal,\nideal]}{(\mideal,\nideal)}.
\]
This choice makes
$\kideal\mideal\nideal=[\mideal,\nideal]^2$ and gives the weight
\[
 (\Norm\kideal)^{-s}
 =\left(\frac{\Norm(\mideal,\nideal)}
              {\Norm[\mideal,\nideal]}\right)^{1/2+\delta}.
\]
It is therefore exactly the kernel in the moving G\'al sum.
The support of $\kideal\mideal\nideal$ lies among primes of norm
at most $y_{\mathcal M}$.
The prime ideal theorem and partial summation give
$\sum_{\Norm\p\leq y}(\Norm\p)^{-1}\leq\log\log y+O_K(1)$
for $y\geq3$.
Using $\log(1+u)\leq u$, we obtain
\begin{equation}\label{eq:h-loss}
 h_K(\kideal\mideal\nideal)
 \geq\prod_{\substack{\Norm\p\leq y_{\mathcal M}\\\p\notin S}}
       (1+(\Norm\p)^{-1})^{-1}
 \geq\exp\{-O_K(\log_3N)\}
 =\exp\{-o(\mathcal V(N))\}.
\end{equation}

Retain only pairs for which $\Norm\kideal\leq X^\eps$.
Since $\eps<1/2$, \eqref{eq:shifted-kernel-small} gives
\[
 \int_X^{2X}w^+_{\kideal,X,A}(t)\,dt
 =XI_W\left(1+O\left(X^{-\rho(1/2-\eps)}\right)\right)
\]
for any fixed $0<\rho<1/2-\delta_0$.
Lemma~\ref{lem:gal-truncation} supplies the sum of their weights.
Since
$\mathcal V(N)=(\sqrt\beta+o(1))\mathcal V(X)$,
we conclude that
\begin{equation}\label{eq:critical-S1}
 S_1(A)\geq XN
 \exp\{(2e^{-A}\sqrt\beta+o(1))\mathcal V(X)\}.
\end{equation}
At $A=0$, both approximate-functional-equation sums are equal.
Keeping both gives the more precise central estimate
\begin{equation}\label{eq:S1-central}
 S_1\geq2\kappa_{\Om}I_WXN
 \exp\{(2+o(1))\mathcal V(N)\}.
\end{equation}
All error terms above are uniform for $0\leq A\leq A_0$.

The lower bound for $S_1(A)$ is positive, so $S_0>0$.
Its quotient by $S_0$ is a weighted average of the real values
$L(s,\chi_{\aideal})$.  Hence
\[
 \max_{\aideal\in\Om(X)}L(s,\chi_{\aideal})
 \geq\frac{S_1(A)}{S_0}
 \geq\exp\{(2e^{-A}\sqrt\beta+o(1))\mathcal V(X)\}.
\]
For every fixed $\beta<1/4$, this holds uniformly on $[0,A_0]$.
Given any $\tau>0$, first take $\beta$ close enough to $1/4$ that
$2\sqrt\beta>1-\tau$, and then take $X$ sufficiently large.  The following choice
realizes the displayed $o(1)$ as a single function of $X$. If one takes a
sequence $\beta_j\uparrow1/4$ and chooses increasing thresholds on
which the estimate for $\beta_j$ is uniform in $0\le A\le A_0$, the
resulting piecewise-constant choice gives a coefficient
$e^{-A}+o(1)$ uniformly on the compact $A$-interval.  This proves the
maximum assertion, including its uniformity.

For abundance, fix $0<c<1$ and choose
\begin{equation}\label{eq:central-beta-choice}
 \frac{c^2}{4}<\beta<\frac14.
\end{equation}
Let
\[
 \mathcal H_{A,c}(X):=\{\aideal\in\Om(X):
 L(s,\chi_{\aideal})\geq e^{ce^{-A}\mathcal V(X)}\}.
\]
Outside this set, the first moment contribution is at most
$e^{ce^{-A}\mathcal V(X)}S_0$.
This upper bound also holds for negative $L$-values.
By \eqref{eq:S0-central}, we have $S_0\leq C XN$ with a fixed
constant $C$.
The discarded contribution divided by the lower bound in
\eqref{eq:critical-S1} is therefore at most
\[
 C\exp\{-[e^{-A}(2\sqrt\beta-c)+o(1)]\mathcal V(X)\}=o(1).
\]
The gap is bounded below by
$e^{-A_0}(2\sqrt\beta-c)>0$, so this estimate is uniform in $A$.
It follows that
\begin{equation}\label{eq:high-mass}
 \sum_{\aideal\in\mathcal H_{A,c}(X)}W(\Norm\aideal/X)
 L(s,\chi_{\aideal})R_{\aideal}^2
 \geq XN\exp\{(2e^{-A}\sqrt\beta+o(1))\mathcal V(X)\}.
\end{equation}

Put $U_K(X):=\exp\{C_K\log X/\log_2X\}$, as in
Lemma~\ref{lem:central-upper}.
The $L$-values on $\mathcal H_{A,c}(X)$ are positive and at most $U_K(X)$.
Cauchy--Schwarz and Lemma~\ref{lem:nf-fourth-moment} give
\begin{align*}
 &\sum_{\aideal\in\mathcal H_{A,c}(X)}W(\Norm\aideal/X)
 L(s,\chi_{\aideal})R_{\aideal}^2\\
 &\quad\leq U_K(X)
 \left(\sum_{\aideal\in\mathcal H_{A,c}(X)}W(\Norm\aideal/X)\right)^{1/2}
 \left(\sum_{\aideal\in\Om}W(\Norm\aideal/X)|R_{\aideal}|^4\right)^{1/2}\\
 &\quad\ll U_K(X)\,\#\mathcal H_{A,c}(X)^{1/2}
 X^{1/2}N X^{o(1)}.
\end{align*}
Comparing this with \eqref{eq:high-mass} and squaring yields
\[
 \#\mathcal H_{A,c}(X)
 \geq XU_K(X)^{-2}X^{-o(1)}
 \exp\{(4e^{-A}\sqrt\beta+o(1))\mathcal V(X)\}
 =X^{1-o(1)}.
\]
Here both $\log U_K(X)$ and $\mathcal V(X)$ are $o(\log X)$. In
particular
\[
 \log\!\left(U_K(X)^{-2}\exp\{O(\mathcal V(X))\}\right)=o(\log X),
\]
which justifies absorbing all remaining factors into $X^{o(1)}$.
This proves the second assertion, uniformly on $[0,A_0]$.

Finally, we choose one threshold approaching the endpoint.
For $j\geq1$, put
\[
 c_j:=1-\frac{3}{10j},\qquad
 \beta_j:=\frac14-\frac1{10j}.
\]
Writing $u=1/(10j)$, we have
$c_j^2-4\beta_j=u(9u-2)<0$.
The fixed-parameter abundance estimate therefore applies to
$(c_j,\beta_j)$.
Choose an increasing sequence $X_j\to\infty$ such that, for all
$X\geq X_j$ and $0\leq A\leq A_0$, it gives
$\#\mathcal H_{A,c_j}(X)\geq X^{1-1/j}$.
For $X_j\leq X<X_{j+1}$, define
\[
 \eta_{A_0}(X):=1-c_j=\frac3{10j}.
\]
Then $\eta_{A_0}(X)\to0$, and the asserted count is at least
$X^{1-1/j}=X^{1-o(1)}$ uniformly on $[0,A_0]$. Notice that this argument
produces one sufficiently slowly decreasing function $\eta_{A_0}$. It
does not assert the conclusion for an arbitrary prescribed function
$\eta(X)\to0$, nor does it give an effective decay rate. Taking $A_0=0$
gives the central endpoint statement with $\eta(X)=\eta_0(X)$.
\end{proof}

\section{Global function fields}\label{sec:function-field}

Let $C/\mathbb F_q$ be a fixed smooth, projective, geometrically connected
curve, where $q$ is odd, and put $F=\mathbb F_q(C)$.  Write $g_C$ for its
genus.  All constants in this section may depend on this curve and on the
fixed place chosen below. The fixed local condition determines the
family in which we prove the prime average. With this estimate, the
resonance argument uses the finite block calculation from
Section~\ref{sec:gal}.

\subsection{The family and its prime averages}

Choose a place $\infty$ of odd degree $d_\infty$.
Such a place exists by the prime-divisor theorem
\cite[Theorem~5.12]{Rosen}.  Put
\[
 \mathscr A=\Gamma(C\setminus\{\infty\},\mathcal O_C),
 \qquad S=\{\infty\},\qquad
 h_{\mathscr A}=|\operatorname{Cl}(\mathscr A)|.
\]
The ring $\mathscr A$ is Dedekind, and its units are $\mathbb F_q^\times$.
We identify its prime ideals with the places different from $\infty$.
If $\mathfrak a=\prod_v\mathfrak p_v^{e_v}$ is a nonzero integral ideal,
define $\deg\mathfrak a=\sum_v e_v\deg v$ and
$|\mathfrak a|=q^{\deg\mathfrak a}$.
Thus sums over ideals are sums over effective divisors away from $\infty$.

Fix a uniformizer $\varpi_\infty$, and define the leading coefficient
\[
 \operatorname{sgn}_\infty(f)
 =\overline{f\varpi_\infty^{-v_\infty(f)}}
 \in\mathbb F_{q^{d_\infty}}^\times
 \qquad(f\in F^\times).
\]
Since $d_\infty$ is odd, a constant is a square in this residue field
if and only if it is a square in $\mathbb F_q$.
For degrees
\begin{equation}\label{eq:ff-admissible-degrees}
 n=d_\infty(2r+1),\qquad r\geq0,
\end{equation}
let $\Om_F(n)$ be the principal prime ideals of $\mathscr A$ of degree $n$.
For $\mathfrak P\in\Om_F(n)$, choose a generator $\pi_{\mathfrak P}$ with
\begin{equation}\label{eq:ff-principal-generator-divisor}
 (\pi_{\mathfrak P})
 =\mathfrak P-\frac n{d_\infty}\infty,
 \qquad
 \operatorname{sgn}_\infty(\pi_{\mathfrak P})
 \in\mathbb F_{q^{d_\infty}}^{\times2}.
\end{equation}
Multiplication by a nonsquare constant changes the square class of the
leading coefficient.  Two generators satisfying this condition differ by
a square constant.  Hence the quadratic extension
$F(\sqrt{\pi_{\mathfrak P}})/F$ is well defined.
Let $\chi_{\mathfrak P}$ be its nontrivial Artin character.
Its coefficients at ramified places are defined to be zero.

This normalization serves two purposes.  It fixes the ramification at
$\infty$, and it makes evaluation at a fixed ideal into a ray class
character of the varying prime.  The latter property gives an error
depending only linearly on the degree of the fixed ideal.

The following prime estimate for geometrically nontrivial characters
gives the family count without using the later orthogonality proposition.

\begin{lemma}\label{lem:ff-geometric-prime-bound}
Let $\theta$ be a geometrically nontrivial finite-order idele class
character of $F$. Then, uniformly for $n\geq1$,
\begin{equation}\label{eq:ff-prime-character-bound}
 \sum_{\substack{\deg\mathfrak P=n\\\mathfrak P\ne\infty}}
       \theta(\mathfrak P)
 \ll_F\frac{(1+\deg\mathfrak f_\theta)q^{n/2}}n.
\end{equation}
\end{lemma}

\begin{proof}
The complete $L$-function of $\theta$ is a polynomial of degree
$D_\theta=2g_C-2+\deg\mathfrak f_\theta$. See
\cite[Theorem~9.24A, p.~141]{Rosen}. To invoke Weil's theorem in a
manifestly geometric form, choose a degree-one divisor $D$ whose support
is disjoint from the conductor and put
$\widetilde\theta=\theta\,\theta(D)^{-\deg}$.
We do not require a degree-one place: by the
prime-divisor theorem, for all sufficiently large $m$ there are places
$P_m$ and $P_{m+1}$ of degrees $m$ and $m+1$, and they may be chosen
outside the finite conductor. Then $D=P_{m+1}-P_m$ has degree one.
The twisted character agrees with $\theta$ on degree-zero idele
classes and satisfies $\widetilde\theta(D)=1$. Every idele class of
degree $m$ is the product of a degree-zero class and $D^m$, so the
image of $\widetilde\theta$ equals its image on the degree-zero subgroup.
Thus the cyclic extension associated with $\widetilde\theta$ has
no nontrivial constant-field quotient and is geometric. See
\cite[Proposition~9.22]{Rosen}. Its character remains nontrivial.
This conclusion concerns the twisted character: geometric
nontriviality alone need not make the original cyclic extension
geometric before removal of its degree component.
Therefore Weil's theorem \cite[Theorem~9.16B, p.~129]{Rosen} gives
inverse roots of $L(u,\widetilde\theta)$ of absolute value $q^{1/2}$.
Undoing the degree twist only rotates these roots, so the inverse roots
$\alpha_{\theta,j}$ of $L(u,\theta)$ have the same absolute value.

Logarithmic differentiation gives
\[
 \sum_{\deg\mathfrak d=n}\Lambda_F(\mathfrak d)\theta(\mathfrak d)
 =-\sum_{j=1}^{D_\theta}\alpha_{\theta,j}^n.
\]
The left side runs over all effective divisors on $C$, with
$\Lambda_F(v^k)=\deg v$ on prime powers and zero otherwise.
Removing prime powers of exponent at least two costs $O_F(q^{n/2})$:
prime squares contribute $O_F(q^{n/2})$, while higher powers contribute
$O_F(nq^{n/3})=O_F(q^{n/2})$. Removing the fixed place $\infty$ costs only $O_F(1)$ (its von
Mangoldt weight is $d_\infty$ when it occurs). Division by $n$ proves
\eqref{eq:ff-prime-character-bound}.
\end{proof}

\begin{proposition}\label{prop:ff-positive-local-component}
For the degrees in \eqref{eq:ff-admissible-degrees}, the following hold.
\begin{enumerate}
\item Each $\chi_{\mathfrak P}$ has conductor
\begin{equation}\label{eq:ff-positive-local-conductor}
 \mathfrak f_{\mathfrak P}=\mathfrak P+\infty.
\end{equation}
In particular, its coefficient at $\infty$ is zero.
\item With $\kappa_F=d_\infty/h_{\mathscr A}>0$, one has
\begin{equation}\label{eq:ff-prime-count}
 \#\Om_F(n)=\kappa_F\frac{q^n}{n}
 +O_F\left(\frac{q^{n/2}}n\right).
\end{equation}
\end{enumerate}
\end{proposition}

\begin{proof}
The valuations of $\pi_{\mathfrak P}$ are odd at $\mathfrak P$ and
$\infty$, and zero elsewhere.  Quadratic Kummer ramification is tame
because $q$ is odd.  The conductor exponent is one at these two places
and zero elsewhere.  This proves the conductor formula.

Extend each character $\xi$ of $\operatorname{Cl}(\mathscr A)$ to
divisors on $C$ by setting $\xi(\infty)=1$.
The extension is trivial on principal divisors, and so is an unramified
idele class character.
Exactly $d_\infty$ of these characters factor through the degree map:
they are $\xi(\mathfrak d)=\zeta^{\deg\mathfrak d}$ with
$\zeta^{d_\infty}=1$.
Indeed, the prime-divisor theorem supplies places of two consecutive
sufficiently large degrees, whose difference is a divisor of degree one.
Thus the degree map on divisor classes is onto $\mathbb Z$, and after
quotienting by the class of $\infty$ the degree map on
$\operatorname{Cl}(\mathscr A)$ has image
$\mathbb Z/d_\infty\mathbb Z$.
Every other character is geometrically nontrivial and satisfies
Lemma~\ref{lem:ff-geometric-prime-bound} with conductor zero. Since
$d_\infty\mid n$, each of the $d_\infty$ degree characters takes the
value one on every prime of degree $n$. Character orthogonality and
\cite[Theorem~5.12]{Rosen} now give
\begin{align*}
 \#\Om_F(n)
 &=\frac1{h_{\mathscr A}}
   \sum_\xi\sum_{\substack{\deg\mathfrak P=n\\\mathfrak P\ne\infty}}
                        \xi(\mathfrak P)\\
 &=\frac{d_\infty}{h_{\mathscr A}}
      \left(\frac{q^n}{n}+O_F(q^{n/2}/n)\right)
       +O_F(q^{n/2}/n).
\end{align*}
This proves \eqref{eq:ff-prime-count}.
\end{proof}

\begin{proposition}\label{prop:ff-orthogonality}
Let $\mathfrak l$ be a nonzero integral ideal of $\mathscr A$ having no
prime factor of degree $n$.  Write
$\mathfrak l=\mathfrak l_0\mathfrak l_1^2$, with $\mathfrak l_0$
squarefree.  Uniformly in $\mathfrak l$,
\begin{equation}\label{eq:ff-orthogonality}
 \sum_{\mathfrak P\in\Om_F(n)}\chi_{\mathfrak P}(\mathfrak l)
 =\kappa_F\frac{q^n}{n}\one_{\mathfrak l_0=\mathscr A}
 +O_F\left(\frac{(1+\deg\mathfrak l_0)q^{n/2}}n\right).
\end{equation}
\end{proposition}

\begin{proof}
We construct a reciprocal character for $\mathfrak l_0$.
Let $I_{\mathscr A}(\mathfrak l_0)$ be the fractional ideals prime to
$\mathfrak l_0$, and let $P_{\mathscr A}^+(\mathfrak l_0)$ consist of
the principal ideals $(a)_{\mathscr A}$ satisfying
\[
 a\equiv1\pmod{\mathfrak l_0},\qquad
 \operatorname{sgn}_\infty(a)
       \in\mathbb F_{q^{d_\infty}}^{\times2}.
\]
The congruence requires $a$ to be a unit at every prime of the modulus.
The quotient
\begin{equation}\label{eq:ff-signed-ray-group}
 G(\mathfrak l_0)
 =I_{\mathscr A}(\mathfrak l_0)/P_{\mathscr A}^+(\mathfrak l_0)
\end{equation}
is finite.  Indeed, it maps onto $\operatorname{Cl}(\mathscr A)$, and
its kernel is a quotient of
\[
 \prod_{v\mid\mathfrak l_0}\mathbb F_{q^{\deg v}}^\times
 \times\bigl(\mathbb F_{q^{d_\infty}}^\times/
              \mathbb F_{q^{d_\infty}}^{\times2}\bigr).
\]
Let $H(\mathfrak l_0)$ be the subgroup represented by principal ideals.
For $I=(a)_{\mathscr A}$, multiply $a$ by a constant to make its leading
coefficient square, and denote the result by $a^\sharp$.
Define
\begin{equation}\label{eq:ff-reciprocal-subgroup-character}
 \lambda_{\mathfrak l_0}(I)
 =\prod_{v\mid\mathfrak l_0}
      \left(\frac{\overline{a^\sharp}_v}{v}\right)_2.
\end{equation}
Here $(\,\cdot\,/v)_2$ is the quadratic character of the residue field.
The choice of $a^\sharp$ is unique up to a square constant.
Thus the formula is well defined and multiplicative.
It is trivial on $P_{\mathscr A}^+(\mathfrak l_0)$.

A character of a subgroup of a finite abelian group extends to the group.
Choose an extension $\rho_{\mathfrak l_0}$ of
$\lambda_{\mathfrak l_0}$ to $G(\mathfrak l_0)$.
As an idele class character, it has conductor dividing
$\mathfrak l_0+\infty$.
The local Artin symbol and the degree assumption give
\begin{equation}\label{eq:ff-direct-reciprocity}
 \chi_{\mathfrak P}(\mathfrak l)
 =\prod_{v\mid\mathfrak l_0}
      \left(\frac{\overline{\pi_{\mathfrak P}}_v}{v}\right)_2
 =\rho_{\mathfrak l_0}(\mathfrak P)
 \qquad(\mathfrak P\in\Om_F(n)).
\end{equation}
Inflate the class group characters to $G(\mathfrak l_0)$.
Their orthogonality yields
\begin{equation}\label{eq:ff-orthogonality-expanded}
 \sum_{\mathfrak P\in\Om_F(n)}\chi_{\mathfrak P}(\mathfrak l)
 =\frac1{h_{\mathscr A}}
 \sum_{\xi\in\widehat{\operatorname{Cl}(\mathscr A)}}
 \sum_{\substack{\deg\mathfrak P=n\\\mathfrak P\ne\infty}}
       (\rho_{\mathfrak l_0}\xi)(\mathfrak P).
\end{equation}

Suppose $\mathfrak l_0\ne\mathscr A$, and fix $v_0\mid\mathfrak l_0$.
Weak approximation gives $a\in F^\times$ with residue $1$ at $\infty$
and at every $v\mid\mathfrak l_0$ other than $v_0$, and a nonsquare
residue at $v_0$.
Then $\lambda_{\mathfrak l_0}((a)_{\mathscr A})=-1$.
Also, $\deg(a)_{\mathscr A}=0$ because $v_\infty(a)=0$. Since
$\xi$ is inflated from the ideal class group, it is trivial on the
principal ideal $(a)_{\mathscr A}$, whereas the chosen extension satisfies
$\rho_{\mathfrak l_0}((a)_{\mathscr A})
 =\lambda_{\mathfrak l_0}((a)_{\mathscr A})=-1$. Hence
\[
 (\rho_{\mathfrak l_0}\xi)((a)_{\mathscr A})=-1.
\]
The class of $(a)_{\mathscr A}$ has degree zero, so every character
$\theta=\rho_{\mathfrak l_0}\xi$ in
\eqref{eq:ff-orthogonality-expanded} is nontrivial on the degree-zero
idele classes. This is precisely geometric nontriviality. In particular,
none of these characters factors through the degree map.
Moreover,
\[
 \deg\mathfrak f_\theta\leq\deg\mathfrak l_0+d_\infty.
\]
Applying Lemma~\ref{lem:ff-geometric-prime-bound} to each
$\theta=\rho_{\mathfrak l_0}\xi$ and using
$\deg\mathfrak f_\theta\leq\deg\mathfrak l_0+d_\infty$ in
\eqref{eq:ff-orthogonality-expanded} proves the required estimate for
nonsquares. If $\mathfrak l_0=\mathscr A$, then $\mathfrak l$ is a
square and, by the hypothesis excluding degree-$n$ prime factors,
$\chi_{\mathfrak P}(\mathfrak l)=1$ for every
$\mathfrak P\in\Om_F(n)$. Hence the left side of
\eqref{eq:ff-orthogonality} is exactly $\#\Om_F(n)$, and
\eqref{eq:ff-prime-count} proves the assertion.
\end{proof}

\subsection{The functional equation and positivity}

Put $d_n=n+d_\infty+2g_C-2$.
This is even on the degrees in \eqref{eq:ff-admissible-degrees}.
Write $L(u,\chi)$ for the polynomial variable, so that
$L(s,\chi)=L(q^{-s},\chi)$ in the complex variable $s$.

\begin{lemma}\label{lem:ff-shifted-afe}
For $\mathfrak P\in\Om_F(n)$, the following hold.
\begin{enumerate}
\item The polynomial has degree $d_n$ and satisfies
\begin{equation}\label{eq:ff-root-plus-functional-equation}
 L(u,\chi_{\mathfrak P})
 =(q^{1/2}u)^{d_n}L((qu)^{-1},\chi_{\mathfrak P}).
\end{equation}
\item For every $\delta\geq0$,
\begin{equation}\label{eq:ff-shifted-afe}
 L(1/2+\delta,\chi_{\mathfrak P})
 =\sum_{\mathfrak a}
   \frac{\chi_{\mathfrak P}(\mathfrak a)}{|\mathfrak a|^{1/2+\delta}}
   V_{d_n,\delta}(\deg\mathfrak a),
\end{equation}
where the sum is over nonzero integral ideals of $\mathscr A$, and
\begin{equation}\label{eq:ff-shifted-weight}
 V_{d,\delta}(r)=
 \begin{cases}
 1+q^{-\delta(d-2r)},&2r<d,\\
 1,&2r=d,\\
 0,&2r>d.
 \end{cases}
\end{equation}
In particular, this weight lies in $[1,2]$ on $2r\leq d$.
\item One has $L(\sigma,\chi_{\mathfrak P})\geq0$ for real
$\sigma\geq1/2$, with strict positivity for $\sigma>1/2$.
\label{eq:ff-positive}
\end{enumerate}
\end{lemma}

\begin{proof}
Let $C_{\mathfrak P}$ be the smooth projective curve with function field
$F(\sqrt{\pi_{\mathfrak P}})$.
Ramification shows that the extension has full constant field $\mathbb F_q$.
Artin factorization gives
\begin{equation}\label{eq:ff-artin-zeta-factorization}
 Z_{C_{\mathfrak P}}(u)=Z_C(u)L(u,\chi_{\mathfrak P}).
\end{equation}
The two curve zeta functions have the same denominator $(1-u)(1-qu)$.
Their numerator polynomials satisfy functional equations with sign $+1$.
Riemann--Hurwitz gives
\[
 2g_{C_{\mathfrak P}}-2=2(2g_C-2)+n+d_\infty,
 \qquad 2(g_{C_{\mathfrak P}}-g_C)=d_n.
\]
Taking the quotient of the two functional equations proves
\eqref{eq:ff-root-plus-functional-equation}.

To obtain the finite sum, write
$L(u,\chi_{\mathfrak P})=\sum_{r=0}^{d_n}a_r u^r$.
Its coefficients are
$a_r=\sum_{\deg\mathfrak a=r}\chi_{\mathfrak P}(\mathfrak a)$,
since the factor at $\infty$ is one.
Comparing coefficients in the functional equation gives
$a_{d_n-r}=q^{d_n/2-r}a_r$.
At $u=q^{-1/2-\delta}$, a pair of terms with $2r<d_n$ is therefore
\[
 a_rq^{-(1/2+\delta)r}
 +a_{d_n-r}q^{-(1/2+\delta)(d_n-r)}
 =a_rq^{-(1/2+\delta)r}
       \{1+q^{-\delta(d_n-2r)}\}.
\]
The middle term occurs once.  This proves the formula.
At $\delta=0$, it is the usual exact central formula.
Compare \cite[Section~3.5, Lemma~1]{AndradeKeating}.

By Weil's theorem,
\[
 L(q^{-1/2}z,\chi_{\mathfrak P})
 =\prod_{j=1}^{d_n}(1-ze^{i\theta_j}).
\]
The coefficients are real.  A conjugate pair contributes
$|1-ze^{i\theta_j}|^2$, while a real factor is $1-z$ or $1+z$.
All these factors are nonnegative for $0\leq z\leq1$, and positive
for $0\leq z<1$.  Take $z=q^{1/2-\sigma}$.
\end{proof}

For comparison with polynomial families, take $F=\mathbb F_q(t)$ and
$\deg\infty=1$. Our convention is $\chi_P(f)=(P/f)$, whereas
\cite{DarbarMaitiFF} uses $(f/P)$ for monic polynomials.
For odd $\deg P=n$, quadratic reciprocity gives
\[
 (P/f)=(-1)^{((q-1)/2)\deg f}(f/P).
\]
The conventions agree when $q\equiv1\pmod4$.
When $q\equiv3\pmod4$, set $P^*(t)=-P(-t)$.
This is an involution on the monic primes of odd degree.
For $\deg f=r$, the substitution $f^*(t)=(-1)^rf(-t)$ permutes
the monic polynomials of degree $r$ and gives
\[
 (f^*/P^*)=(-1)^r(f/P).
\]
Thus the coefficient sums for our character at $P$ equal those for
the convention in \cite{DarbarMaitiFF} at $P^*$.
The two conventions give the same multiset of $L$-values on the family,
so the comparison of lower-bound constants is unaffected.

\subsection{A G\'al sum with prescribed prime degrees}

The prime-divisor theorem has the form
\begin{equation}\label{eq:ff-prime-divisor-theorem}
 \#\{\mathfrak p:\deg\mathfrak p=j\}
 =\frac{q^j}{j}+O_F(q^{j/2}/j)
\end{equation}
by \cite[Theorem~5.12]{Rosen}.
We use one exact degree per block.  All weights in a block are then equal,
which avoids any loss from the discreteness of the possible norms.

\begin{proposition}\label{prop:ff-gal}For every sufficiently large integer $N$, there is a set $\mathcal M_F$
of $N$ squarefree integral ideals of $\mathscr A$ with the following
properties.
\begin{enumerate}
\item Uniformly for $A$ in a fixed compact subset of $[0,\infty)$,
\begin{equation}\label{eq:ff-moving-gal}
 \sum_{\mathfrak m,\mathfrak r\in\mathcal M_F}
 \left(\frac{|(\mathfrak m,\mathfrak r)|}
 {|[\mathfrak m,\mathfrak r]|}\right)^{1/2+A/\log_2N}
 \geq N\exp\{(2e^{-A}+o(1))\mathcal V(N)\}.
\end{equation}
\item Its prime factors have norm at most $(\log N)^{1+o(1)}$, and
\begin{equation}\label{eq:ff-degree-size}
 \max_{\mathfrak m\in\mathcal M_F}\deg\mathfrak m
 \ll_F\frac{\log N\log_2N}{\log_3N},
 \qquad E_\square(\mathcal M_F)\leq N^{2+o(1)}.
\end{equation}
\item If $N\leq q^{n/4}$, restricting the sum to
\begin{equation}\label{eq:ff-gal-trunc}
 \deg\frac{[\mathfrak m,\mathfrak r]}{(\mathfrak m,\mathfrak r)}
 \leq\vartheta n
\end{equation}
for fixed $\vartheta>0$ preserves the lower bound.
If $N=X_n^{\beta+o(1)}$ with fixed $\beta>0$, the exponent
$A/\log_2N$ may also be replaced by $A/\log_2X_n$.
\end{enumerate}
\end{proposition}

\begin{proof}
Fix $0<\gamma<1$ and $0<\eta<1/2$ for now, and abbreviate
$T_i=\log_iN$.
Set
\[
 D_N=\lceil\log_q(T_1T_2)\rceil,
 \qquad K_N=\lfloor T_2^\gamma\rfloor.
\]
Let $\mathscr P_k$ be the primes of $\mathscr A$ of degree $D_N+k$,
where $1\leq k\leq K_N$, and put $P_k=|\mathscr P_k|$.
Since $K_N=o(D_N)$, \eqref{eq:ff-prime-divisor-theorem} gives, uniformly
over the blocks,
\begin{equation}\label{eq:ff-layer-count}
 P_k=c_Nq^kT_1(1+o(1)),\quad c_N\asymp_q1,
 \qquad P_kq^{-(D_N+k)}=\frac{\log q+o(1)}{T_2}.
\end{equation}

For fixed $a>0$, take all products containing at most
$J_k(a)=\lfloor aT_1/(k^2T_3)\rfloor$ primes in block $k$.
Denote the product of these families by $\mathcal M_F(a)$.
The binomial estimate \eqref{eq:binomial-tail-entropy} gives
\begin{equation}\label{eq:ff-layer-entropy}
 \log|\mathcal M_F(a)|=(a\gamma\log q+o(1))T_1.
\end{equation}
Here is the calculation of its leading term.
Uniformly in $k$,
\[
 \log(P_k/J_k(a))
 =k\log q+2\log k+T_4+O_{q,a}(1),
 \qquad J_k(a)/P_k=o(1).
\]
The first term contributes
$\log q\sum_k kJ_k(a)=(a\gamma\log q+o(1))T_1$.
The remaining terms are bounded by
\[
 O_{q,a}\left(\frac{T_1(1+T_4)}{T_3}
       +K_NT_2+\frac{T_1}{T_3^2}\right)=o(T_1).
\]
This includes the floor errors and the terms
$\sum_k J_k(a)^2/P_k\ll_q T_1/T_3^2$.

Put $a_\pm=(1\pm\eta)/(\gamma\log q)$.
Starting with the quotas $J_k(a_-)$, increase them one at a time toward
$J_k(a_+)$.  Stop immediately before the size first exceeds $N$.
The resulting product family $\mathcal M_*$ has size $M\leq N$ and
quotas $J_k$ between these two bounds.  Comparing consecutive binomial
tails gives
\begin{equation}\label{eq:ff-near-N}
 N/M\leq1+\max_kP_k\leq T_1^{1+o(1)}.
\end{equation}

We apply Lemma~\ref{lem:squarefree-block} to the shifted weights
\[
 w_k(A)=q^{-(1/2+A/T_2)(D_N+k)},
 \qquad H_k(A)=P_kw_k(A).
\]
Since $(D_N+k)\log q=T_2+T_3+k\log q+O_q(1)$, one has
\begin{equation}\label{eq:ff-moving-local}
 w_k(A)=e^{-A+o(1)}q^{-(D_N+k)/2}
\end{equation}
uniformly in $k$ and in bounded $A$.
Choose
\[
 r_k(A)=\left\lfloor
 \frac{e^{-A}\sqrt{a_-\log q}}k
 \sqrt{\frac{T_1}{T_2T_3}}\right\rfloor.
\]
For $A$ in a fixed compact interval, these integers tend to infinity
uniformly in $k$.
The remaining hypotheses of the block lemma follow from
\[
 \frac{r_k(A)}{J_k}
 \ll\frac{K_N\sqrt{T_3}}{\sqrt{T_1T_2}}=o(1),
 \qquad \frac{J_k}{P_k}\ll_q T_3^{-1}=o(1),
\]
and
\[
 \frac{J_k+r_k(A)^2}{P_k}
 \ll_q\frac{q^{-k}}{k^2T_3}
       +\frac{q^{-k}}{k^2T_2T_3}=o(1).
\]
Moreover, \eqref{eq:ff-layer-count} implies
\[
 \frac{e^2J_kH_k(A)^2}{P_kr_k(A)^2}
 \geq e^2(1+o(1)).
\]
Let $S_s(\mathcal B)$ denote the G\'al sum over $\mathcal B$ with
exponent $s$.  Factoring it over the blocks gives
\begin{align}\label{eq:ff-gal-intermediate}
 \log\frac{S_{1/2+A/T_2}(\mathcal M_*)}{M}
 &\geq(2+o(1))\sum_kr_k(A)\notag\\
 &=\bigl(2e^{-A}\sqrt{\gamma(1-\eta)}+o(1)\bigr)\mathcal V(N).
\end{align}
The factor $\log q$ in the block density cancels the factor $\log q$
in the cardinality estimate.  This explains why the leading constant
does not depend on $q$.

To obtain exactly $N$ elements, put
$h=\lfloor\log(N/M)/\log2\rfloor$.
Choose $h$ new primes, each of degree $O_F(\log(h+2))$, and take all
products of these primes.  Call this cube $\mathcal C_h$.
Its support is disjoint from the original blocks, and
\[
 \frac{S_s(\mathcal C_h)}{|\mathcal C_h|}
 =\prod_{\mathfrak q\text{ new}}(1+|\mathfrak q|^{-s})\geq1.
\]
Thus $\mathcal M_1=\mathcal C_h\mathcal M_*$ has size in $(N/2,N]$,
and
\begin{equation}\label{eq:ff-cube-gal-factorization}
 \frac{S_s(\mathcal M_1)}{|\mathcal M_1|}
 \geq\frac{S_s(\mathcal M_*)}{M}.
\end{equation}
Add elements from $\mathcal M_F(a_+)\setminus\mathcal M_1$ until the
size is $N$.
There are enough, since $|\mathcal M_F(a_+)|>N$ and its intersection
with $\mathcal M_1$ is $\mathcal M_*$.
Positivity of the kernel shows that this completion costs a factor
of at most two in \eqref{eq:ff-gal-intermediate}.

The energy estimate survives the completion.
Lemma~\ref{lem:hamming-energy} gives
\begin{align}\label{eq:ff-upper-energy}
 \log\frac{E_\square(\mathcal M_F(a_+))}{|\mathcal M_F(a_+)|^2}
 &\ll\sum_k\{J_k(a_+)+J_k(a_+)^2/P_k+\log(J_k(a_+)+1)\}\notag\\
 &\ll_F T_1/T_3+T_1/T_3^2+K_NT_2=o(T_1).
\end{align}
Every completed element lies in $\mathcal C_h\mathcal M_F(a_+)$.
Its energy is
$2^{3h}E_\square(\mathcal M_F(a_+))\leq N^{2+2\eta+o(1)}$,
because $2^h\leq T_1^{1+o(1)}$.

A diagonal choice of the parameters gives the limiting constant.
For $j\geq3$, take $\gamma_j=1-1/j$ and $\eta_j=1/j$.
Choose increasing thresholds $N_j$ so that all preceding errors are
at most $1/j$ for $N\geq N_j$, uniformly for $0\leq A\leq j$.
Require also
\[
 E_\square(\mathcal C_h\mathcal M_F(a_+))\leq N^{2+5/j},
 \qquad T_2^{-1/j}\leq1/j.
\]
Use these parameters on $N_j\leq N<N_{j+1}$.
Then \eqref{eq:ff-gal-intermediate} proves \eqref{eq:ff-moving-gal},
and the energy is $N^{2+o(1)}$.
The extra threshold ensures $K_N/D_N=o(1)$ along this choice.
Consequently every original prime has
\[
 \log|\mathfrak p|=(1+o(1))\log_2N,
\]
and the auxiliary primes have smaller degree.
An element of the upper family has degree at most
\[
 \sum_kJ_k(a_+)(D_N+k)
 \ll_F\frac{T_1}{T_3}
        \left(D_N\sum_{k\geq1}k^{-2}+\sum_{k\leq K_N}k^{-1}\right)
 \ll_F\frac{T_1T_2}{T_3}.
\]
The auxiliary primes add $O_F(T_2T_3)$.
This proves all size assertions.

For the truncation, let $\mathscr Q_N$ contain all supporting primes.
Its largest degree is $(1+o(1))\log_q\log N$.
Put $E=\deg([\mathfrak m,\mathfrak r]/(\mathfrak m,\mathfrak r))$.
When $E>\vartheta n$, we have
\[
 q^{-E/2}\leq q^{-\vartheta n/6}q^{-E/3}.
\]
For fixed $\mathfrak m$, enlarge the sum over $\mathfrak r$ to all
squarefree products on $\mathscr Q_N$.
The resulting bound is
\[
 q^{-\vartheta n/6}
 \prod_{\mathfrak p\in\mathscr Q_N}(1+|\mathfrak p|^{-1/3})
 \leq q^{-\vartheta n/6}
       \exp\{(\log N)^{2/3+o(1)}\}.
\]
Summing over $\mathfrak m$ gives $o(N)$ when $N\leq q^{n/4}$.
The shifted kernel is at most the central kernel, so the same bound
applies uniformly in $A\geq0$.
Finally, if $\log N=(\beta+o(1))\log X_n$, put
$B=A\log_2N/\log_2X_n$.
Then $A/\log_2X_n=B/\log_2N$ and $B=A+o(1)$ uniformly for bounded $A$.
Applying the compact-uniform estimate already proved at $B$ gives
the stated replacement of the shift.
\end{proof}

\subsection{A uniform upper bound}

The fourth moment controls the weights placed on a small exceptional set.
To turn that control into a count of large values, we also need an upper
bound for each $L$-value.  The following subpower bound is sufficient.

\begin{lemma}\label{lem:ff-upper}
Uniformly for $\mathfrak P\in\Om_F(n)$ and real $\sigma\geq1/2$,
\begin{equation}\label{eq:ff-uniform-upper}
 0\leq L(\sigma,\chi_{\mathfrak P})
 \leq\exp\{O_F(n\log_2n/\log n)\}=X_n^{o(1)}.
\end{equation}
\end{lemma}

\begin{proof}
Positivity was proved in Lemma~\ref{lem:ff-shifted-afe}.
Write
$Q(z)=L(q^{-1/2}z,\chi_{\mathfrak P})=
\prod_{j=1}^{d_n}(1-ze^{i\theta_j})$.
For an integer $M\geq3$, set $\rho=\exp(-\log M/M)$.
The inequality
\[
 \frac{|1-e^{i\theta}|}{|1-\rho e^{i\theta}|}
 \leq\frac2{1+\rho}
\]
follows by maximizing
$2(1-\cos\theta)/((1-\rho)^2+2\rho(1-\cos\theta))$.
Expand $\log|1-\rho e^{i\theta}|$ into its absolutely convergent series.
Its tail satisfies
\[
 \sum_{k>M}\frac{\rho^k}k
 \leq\frac{\rho^{M+1}}{(M+1)(1-\rho)}
 \ll\frac1{M\log M}.
\]
Since $\log(2/(1+\rho))\ll\log M/M$, we obtain
\begin{equation}\label{eq:ff-log-majorant}
 \log|1-e^{i\theta}|
 \leq-\sum_{k\leq M}\frac{\rho^k\cos(k\theta)}k
      +O(\log M/M).
\end{equation}
The inequality remains valid when the left side is $-\infty$.

Logarithmic differentiation of the Euler product gives
\[
 \sum_{j=1}^{d_n}e^{ik\theta_j}
 =-q^{-k/2}\sum_{\deg\mathfrak a=k}
        \Lambda_F(\mathfrak a)\chi_{\mathfrak P}(\mathfrak a)
 \ll_F q^{k/2}.
\]
The last estimate follows from \eqref{eq:ff-prime-divisor-theorem}.
Apply \eqref{eq:ff-log-majorant} at $\theta_j+\phi$ and sum over $j$.
Uniformly in real $\phi$,
\[
 \log|Q(e^{i\phi})|
 \ll_F\frac{n\log M}{M}+
       \sum_{k\leq M}\frac{q^{k/2}}k
 \ll_F\frac{n\log M}{M}+\frac{q^{M/2}}M.
\]
Take $M=\lfloor\log_qn\rfloor$, adjusting the finitely many small
values of $n$.
The maximum-modulus principle extends this bound to $|z|\leq1$.
Substitute $z=q^{1/2-\sigma}$ to obtain the assertion.
\end{proof}

\subsection{Proof of Theorem~\ref{thm:function-field}}

\begin{proof}
We treat the center and the critical window together.
Fix $A_0\geq0$, let $0\leq A\leq A_0$, and write
$s=1/2+A/\log_2X_n$.
Fix $0<\beta<1/4$ and set $N=\lfloor X_n^\beta\rfloor$.
Choose the set in Proposition~\ref{prop:ff-gal}, and define
\[
 R_{\mathfrak P}=\sum_{\mathfrak m\in\mathcal M_F}
                    \chi_{\mathfrak P}(\mathfrak m),
 \quad T_0=\sum_{\mathfrak P\in\Om_F(n)}R_{\mathfrak P}^2,
 \quad T_4=\sum_{\mathfrak P\in\Om_F(n)}R_{\mathfrak P}^4,
\]
and
\[
 T_{1,A}=\sum_{\mathfrak P\in\Om_F(n)}
                L(s,\chi_{\mathfrak P})R_{\mathfrak P}^2.
\]
The supporting primes of the resonator have degree $O_F(\log n)$.
Thus no twisting ideal formed from its elements has a prime factor of
degree $n$, once $n$ is large.

Since the resonator ideals are squarefree, their product is a square
exactly when they are equal.
Proposition~\ref{prop:ff-orthogonality} therefore gives
\begin{equation}\label{eq:ff-T0}
 T_0=\kappa_F\frac{X_n}{n}N
       +O_F(X_n^{1/2}N^2n^{O(1)})
     =(\kappa_F+o(1))\frac{X_n}{n}N.
\end{equation}
For four resonator ideals, the square terms are counted by
$E_\square(\mathcal M_F)$.
Consequently
\begin{equation}\label{eq:ff-T4}
 T_4=\kappa_F\frac{X_n}{n}E_\square(\mathcal M_F)
       +O_F(X_n^{1/2}N^4n^{O(1)})
     \leq\frac{X_n}{n}N^{2+o(1)}.
\end{equation}
The relative error is $X_n^{-1/2+2\beta+o(1)}=o(1)$.

Insert \eqref{eq:ff-shifted-afe} into $T_{1,A}$.
There are $O_F(q^r)$ ideals of degree $r$, as follows from the
rationality of the curve zeta function.
Only $r\leq d_n/2=n/2+O_F(1)$ occurs.
For large $n$, these ideals also have no prime factor of degree $n$.
The squarefree part of each twisting ideal has degree at most
$r+2\max_{\mathfrak m\in\mathcal M_F}\deg\mathfrak m=n^{O(1)}$.
Since the weights are at most two and $s\geq1/2$, the total
orthogonality error is
\begin{align}\label{eq:ff-T1-error}&\ll_F X_n^{1/2}N^2n^{O(1)}
                  \sum_{r\leq d_n/2}q^{r/2}\notag\\
 &\ll_F X_n^{3/4}N^2n^{O(1)}
       =o(X_nN/n).
\end{align}
This uses $\beta<1/4$.

For each pair $\mathfrak m,\mathfrak r$, retain the ideal
$\mathfrak a=[\mathfrak m,\mathfrak r]/(\mathfrak m,\mathfrak r)$.
It satisfies $\mathfrak a\mathfrak m\mathfrak r
=[\mathfrak m,\mathfrak r]^2$.
Restrict to $\deg\mathfrak a\leq\vartheta n$ with fixed
$0<\vartheta<1/2$.
The polynomial weight is then at least one.
Every other square main term is also nonnegative, so it may be omitted.
The truncated G\'al bound now gives
\begin{equation}\label{eq:ff-T1}
 T_{1,A}\geq\frac{X_n}{n}N
   \exp\{(2e^{-A}\sqrt\beta+o(1))\mathcal V_F(n)\}.
\end{equation}
The fixed factor $\kappa_F$ is absorbed in the exponent error.
All estimates are uniform for $0\leq A\leq A_0$.
Dividing by \eqref{eq:ff-T0} proves the maximum bound with coefficient
$2e^{-A}\sqrt\beta$.  Letting $\beta\uparrow1/4$ through a fixed
sequence and choosing admissible degree thresholds diagonally gives the
coefficient $e^{-A}+o(1)$ uniformly for $0\le A\le A_0$.

To count large values, fix $0<c<1$ and choose $c^2/4<\beta<1/4$.
Let
\[
 \mathcal H_{F,A,c}(n)
 =\{\mathfrak P\in\Om_F(n):
 L(s,\chi_{\mathfrak P})\geq
             \exp(c e^{-A}\mathcal V_F(n))\}.
\]
The complement contributes at most
$\exp(c e^{-A}\mathcal V_F(n))T_0$ to $T_{1,A}$.
By \eqref{eq:ff-T0} and \eqref{eq:ff-T1}, this is $o(T_{1,A})$,
uniformly in $A$, since $2\sqrt\beta-c>0$.
Thus Cauchy--Schwarz and Lemma~\ref{lem:ff-upper} give
\[
 (1-o(1))T_{1,A}
 \leq X_n^{o(1)}\#\mathcal H_{F,A,c}(n)^{1/2}T_4^{1/2}.
\]
It follows that
\[
 \#\mathcal H_{F,A,c}(n)
 \geq\frac{T_{1,A}^2}{X_n^{o(1)}T_4}
 \geq\frac{X_n}{n}X_n^{-o(1)}
 =X_n^{1-o(1)}.
\]

Finally, take
$\beta_j=1/4-1/(10j)$ and $c_j=1-3/(10j)$.
Then $c_j^2<4\beta_j$.
For each fixed $j$, the preceding estimates hold beyond an admissible
degree $n_j$, uniformly for $A\in[0,A_0]$.
Increase the thresholds so that the large-value count is at least
$X_n^{1-1/j}$ whenever $n\geq n_j$.
Use the $j$th choice for $n_j\leq n<n_{j+1}$ and put
$\eta_{F,A_0}(n)=1-c_j$ there.
This gives the endpoint count and hence the endpoint maximum.
Taking $A_0=0$ gives the central assertions.
\end{proof}

\section{Large values in the open strip}\label{sec:strip}

For Theorem~\ref{thm:strip}, we use a finite Euler product to favour
characters which take the value $1$ at small primes.
The same choice of coefficients gives the quadratic Dirichlet results in
\cite[Theorems~4 and~5]{DarbarMaiti}.  For related resonance arguments,
see \cite{Lamzouri,AistleitnerMahatabMunschPeyrot}.
We write $\pi_K(y)$ for the number of prime ideals of norm at most $y$.

\begin{proof}[Proof of Theorem~\ref{thm:strip}]
Fix $1/2<\sigma<1$ and $0<b<1$.  Choose $a>0$, to be specified below, and
put $Y=a\log X\log_2X$.  Define a completely multiplicative function $r$ by
\[
 r_{\p}=b\quad(\Norm\p\leq Y,\ \p\notin S),
 \qquad r_{\p}=0\quad\text{otherwise}.
\]
For each member of the family, put
\[
 R_{\aideal}=\prod_{\substack{\Norm\p\leq Y\\\p\notin S}}
 (1-b\chi_{\aideal}(\p))^{-1}
 =\sum_{\mideal}r_{\mideal}\chi_{\aideal}(\mideal).
\]
The series converges absolutely, since it is a product of finitely many
convergent geometric series.  Fix a nonzero smooth function $W$, with
$0\leq W\leq1$, supported in $(1,2)$, and write $I_W=\int W(t)\,dt$.

Let $y=(\log X)^B$, with $B$ as in Lemma~\ref{lem:short-log}. We may
enlarge $B$ if necessary and assume $B>1$, so that $Y<y$ for all
sufficiently large $X$. Put
\[
 P_\sigma(\chi)=\sum_{\Norm\p\leq y}
       \frac{\chi(\p)}{(\Norm\p)^\sigma},\qquad
 S_2=\sum_{\aideal\in\Om}W(\Norm\aideal/X)R_{\aideal}^2.
\]
We will show that the average of $P_\sigma$ with weight $WR^2$ is large.
Lemma~\ref{lem:short-log} then gives the same lower bound for $\log L$,
up to a bounded error.  GRH and positivity for real $s>1$ imply that
$L(s,\chi_{\aideal})>0$ for $1/2<s\leq1$, so this logarithm is real.

The square terms in the expansion of $R^2$ can be computed prime by prime.
The two relevant geometric sums are
\begin{align}
 D_b&=\sum_{\substack{i,j\geq0\\i+j\ \mathrm{even}}}b^{i+j}
     =\frac{1+b^2}{(1-b^2)^2},\label{eq:strip-local-even}\\
 N_b&=\sum_{\substack{i,j\geq0\\i+j\ \mathrm{odd}}}b^{i+j}
     =\frac{2b}{(1-b^2)^2}.\label{eq:strip-local-odd}
\end{align}
The even sum occurs in $R^2$.  Inserting $\chi(\p)$ changes even parity
to odd parity.  A ramified prime contributes $1$ to the first sum and
$0$ to the second.  Thus, with
\[
 D_X=\prod_{\substack{\Norm\p\leq Y\\\p\notin S}}
       \bigl(1-h_K(\p)+h_K(\p)D_b\bigr),
 \qquad \theta(b)=\log D_b,
\]
Proposition~\ref{prop:orthogonality} gives
\begin{equation}\label{eq:strip-S2}
 S_2=\kappa_{\Om}I_WXD_X
 +O\bigl(X^{1/2+\eps-2a\log(1-b)+o(1)}\bigr),
 \qquad D_X=X^{a\theta(b)+o(1)}.
\end{equation}
Here and below $\eps>0$ is fixed and sufficiently small.

We justify summing the error in this application.  For every ideal
$\mideal$, the factor $G_\eps((\mideal)_1)$ is at most
$\prod_{\p\mid\mideal}(1+(\Norm\p)^{-1/2-\eps/2})$.
Its sum against $r_{\mideal}$ is bounded by
\[
 \prod_{\substack{\Norm\p\leq Y\\\p\notin S}}
 \left(1+\frac{b}{1-b}
          (1+(\Norm\p)^{-1/2-\eps/2})\right)
 =X^{-a\log(1-b)+o(1)}.
\]
The squarefree part of a product of two resonator ideals has logarithmic
norm $O_K(Y)$, so its $F_\eps$ factor is $X^{o(1)}$.
The same support majorant bounds the $G_\eps$ factor of their product.
This proves absolute summability and the error in \eqref{eq:strip-S2}.
Finally, $1-h_K(\p)+h_K(\p)D_b=D_b(1+O_b((\Norm\p)^{-1}))$.
The prime ideal theorem and Mertens' estimate therefore give the stated
size of $D_X$.

After insertion of a prime in the resonator support, the local quotient is
\begin{equation}\label{eq:strip-local-quotient}
 \frac{h_K(\p)N_b}{1-h_K(\p)+h_K(\p)D_b}
 =\frac{2b}{1+b^2}+O_b((\Norm\p)^{-1}).
\end{equation}
This exact parity count is the sharpened local factor mentioned after the
corresponding calculation in \cite{DarbarMaiti}. It is slightly stronger
than replacing the quotient by $b+o(1)$ unless $b$ is close to one.
A prime outside the support gives no square term. Primes in $S$ have
zero character value. Reinserting them in the prime sum below costs $O(1)$.
Consequently,
\begin{align}
 &\sum_{\aideal\in\Om}W(\Norm\aideal/X)
       P_\sigma(\chi_{\aideal})R_{\aideal}^2\notag\\
 &\quad=\kappa_{\Om}I_WXD_X
 \left(\frac{2b}{1+b^2}\sum_{\Norm\p\leq Y}(\Norm\p)^{-\sigma}
       +O_{K,\sigma,b}(1)\right)
 +O\bigl(X^{1/2+\eps-2a\log(1-b)+o(1)}\bigr).
 \label{eq:strip-S1}
\end{align}
The local error is summable because
$\sum_{\p}(\Norm\p)^{-1-\sigma}<\infty$.
For the average error, the inserted prime adds at most $\log y$ to the
logarithmic norm of the squarefree part.  The additional sum of
$(\Norm\p)^{-\sigma}(1+(\Norm\p)^{-1/2-\eps/2})$ over $\Norm\p\leq y$
is $X^{o(1)}$.  The preceding majorant therefore applies unchanged.

Since
\[
 \theta(b)+2\log(1-b)
 =\log\frac{1+b^2}{(1+b)^2}=-\frac{\alpha(b)}2,
\]
both errors are negligible whenever
\begin{equation}\label{eq:a-condition}
 a<\frac1{\alpha(b)},
 \qquad 0<\eps<\frac{1-a\alpha(b)}2.
\end{equation}
Dividing by the common main term in \eqref{eq:strip-S2} gives
\[
 \frac{\sum W(\Norm\aideal/X)P_\sigma(\chi_{\aideal})R_{\aideal}^2}{S_2}
 =\left(\frac{2b\,a^{1-\sigma}}{(1+b^2)(1-\sigma)}+o(1)\right)
       \frac{(\log X)^{1-\sigma}}{(\log_2X)^\sigma}.
\]
We used partial summation in the prime ideal theorem to obtain
$\sum_{\Norm\p\leq Y}(\Norm\p)^{-\sigma}
\sim Y^{1-\sigma}/((1-\sigma)\log Y)$.
For every fixed $a<1/\alpha(b)$ we may first let $X\to\infty$ with
$\eps$ chosen as in \eqref{eq:a-condition}. The resulting lower bound is
valid for every such fixed $a$. Letting $a\uparrow1/\alpha(b)$ afterwards
(or using a diagonal choice of $a=a(X)$) proves the maximum assertion.

For the frequency assertion, put $a_0=1/\alpha(b)-\eta>0$ and write
\[
 H_\sigma(X)=(\log X)^{1-\sigma}(\log_2X)^{-\sigma}.
\]
Choose $a=a_0+\rho_X$, where $\rho_X\to0^+$ so slowly that
$\rho_XH_\sigma(X)\to\infty$, and also
$\rho_X<\tfrac12(1/\alpha(b)-a_0)=\eta/2$ for large $X$.
Thus $a$ remains a fixed positive distance below $1/\alpha(b)$, and the
error estimates above are uniform for $a$ in a compact interval around
$a_0$. Since the main coefficient is a positive constant times
$a^{1-\sigma}$, the mean at $a_0+\rho_X$ exceeds its value at $a_0$ by
$\gg_{\sigma,b,\eta}\rho_XH_\sigma(X)$. Hence, after slowing $\rho_X$
if necessary, the weighted mean exceeds
$\tau_{\sigma,\eta}(X)+C_\sigma+1$, where $C_\sigma$ bounds the error in
Lemma~\ref{lem:short-log}.
Let $\mathcal E\subseteq\Om(X)$ be the set where
$\log L(\sigma,\chi_{\aideal})>\tau_{\sigma,\eta}(X)$.
Outside $\mathcal E$ we have
$P_\sigma\leq\tau_{\sigma,\eta}(X)+C_\sigma$.
Subtracting this contribution from the weighted mean gives
\[
 \sum_{\aideal\in\mathcal E}W(\Norm\aideal/X)R_{\aideal}^2
 \geq S_2X^{-o(1)},
\]
because $|P_\sigma|+|\tau_{\sigma,\eta}(X)|+C_\sigma=X^{o(1)}$.
On the other hand,
\[
 R_{\aideal}^2\leq(1-b)^{-2\pi_K(Y)}
       =X^{-2a\log(1-b)+o(1)}.
\]
Combining these bounds with \eqref{eq:strip-S2} and
$\#\Om(X)=X^{1+o(1)}$, we obtain
\[
 \frac{\#\mathcal E}{\#\Om(X)}
 \geq X^{a\theta(b)+2a\log(1-b)+o(1)}
 =X^{-\frac12(1-\eta\alpha(b))+o(1)}.
\]
\end{proof}

\section{Large values at \texorpdfstring{$1$}{1}}\label{sec:one}

At $s=1$, the small primes determine both the leading factor and the
constant term.  We use the coefficients from
\cite[Theorems~2 and~3]{DarbarMaiti}.  They approach $1$ at each fixed
prime, which makes the resonating Euler product close to Mertens' product.
The distribution of quadratic values at $1$ is studied in
\cite{GranvilleSoundararajan,MontgomeryVaughan}.

The following estimate justifies the infinite series used in the proof.
It also controls products whose individual squarefree parts overlap.

\begin{lemma}\label{lem:euler-error}
Fix $B,c>0$, and put $y=(\log X)^B$ and
$z=c^{-1}\log X\log_2X$.
Let $a_{\lideal}$ be the indicator that all prime factors of $\lideal$
have norm at most $y$.  Define a completely multiplicative function $r$ by
\[
 r_{\p}=1-\Norm\p/z\quad(\Norm\p<z,\ \p\notin S),
 \qquad r_{\p}=0\quad\text{otherwise}.
\]
For every fixed $0<\eps<1/2$, the following estimates hold:
\begin{enumerate}
\item
\begin{equation}\label{eq:l-error-sum}
 \sum_{\lideal}\frac{a_{\lideal}}{\Norm\lideal}
 F_\eps((\lideal)_0)G_\eps((\lideal)_1)=X^{o(1)}.
\end{equation}
\item
\begin{equation}\label{eq:m-error-sum}
 \sum_{\mideal}r_{\mideal}
 F_\eps((\mideal)_0)G_\eps((\mideal)_1)=X^{1/c+o(1)}.
\end{equation}
\item Both estimates remain valid after replacing $G_\eps((\nideal)_1)$
by $\prod_{\p\mid\nideal}(1+(\Norm\p)^{-1/2-\eps/2})$.
In particular,
\[
 \sum_{\lideal,\mideal,\nideal}
 \frac{a_{\lideal}r_{\mideal}r_{\nideal}}{\Norm\lideal}
 F_\eps((\lideal\mideal\nideal)_0)
 G_\eps((\lideal\mideal\nideal)_1)
 \leq X^{2/c+o(1)}.
\]
\end{enumerate}
\end{lemma}

\begin{proof}
We prove the stronger bounds in the third assertion.
In the first sum, the ideals of norm at most $X$ have $F_\eps=X^{o(1)}$.
The remaining Euler product satisfies
\[
 \sum_{\lideal}\frac{a_{\lideal}}{\Norm\lideal}
 \prod_{\p\mid\lideal}(1+(\Norm\p)^{-1/2-\eps/2})
 \ll_{K,\eps}\prod_{\Norm\p\leq y}(1-(\Norm\p)^{-1})^{-1}
 \ll_{K,B,\eps}\log y.
\]
For $\Norm\lideal>X$, put $\delta_X=(\log X)^{-\eps/3}$.
The definition of $F_\eps$ and the maximal-order bound used in
Lemma~\ref{lem:FG-uniform} imply, uniformly in this range,
\[
 F_\eps((\lideal)_0)
 \prod_{\p\mid\lideal}(1+(\Norm\p)^{-1/2-\eps/2})
 \leq(\Norm\lideal)^{\delta_X}.
\]
Indeed, their logarithms are bounded respectively by
\[
 O_{K,\eps}((\log\Norm\lideal)^{1-\eps})
 \quad\text{and}\quad
 O_{K,\eps}((\log\Norm\lideal)^{1/2-\eps/2}).
\]
The tail is therefore at most
\[
 \prod_{\Norm\p\leq y}
 (1-(\Norm\p)^{-1+\delta_X})^{-1}=X^{o(1)}.
\]
To see the last estimate, note that $\delta_X\log y=o(1)$.
The logarithm of this product is $O_K(\log\log y+1)$ by Mertens' estimate.

Every prime factor of a resonator ideal has norm below $z$.
Thus its squarefree part has logarithmic norm $O_K(z)$ and
$F_\eps=X^{o(1)}$ uniformly.  The other factor sums to
\begin{align*}
 &\prod_{\substack{\Norm\p<z\\\p\notin S}}
 \left(1+\frac{r_{\p}}{1-r_{\p}}
       (1+(\Norm\p)^{-1/2-\eps/2})\right)\\
 &\qquad=\exp\left(\sum_{\substack{\Norm\p<z\\\p\notin S}}
             \log\frac z{\Norm\p}+o(\log X)\right)
 =X^{1/c+o(1)}.
\end{align*}
Here $\sum_{\Norm\p<z}(\Norm\p)^{-1/2-\eps/2}=o(\log X)$, and
\[
 \sum_{\Norm\p<z}\log\frac z{\Norm\p}
 =\int_2^z\frac{\pi_K(t)}t\,dt
 =\frac z{\log z}(1+o(1))
 =\frac{\log X}{c}(1+o(1)).
\]
The sum without either error factor has the same exponent, which proves
the matching lower bound in \eqref{eq:m-error-sum}.

For the product assertion, its squarefree support is contained in the
union of the individual squarefree supports.
If $a,b,c$ are the norms of the three individual squarefree parts,
then the norm of the product's squarefree part is at most $abc$.
The inequality $2+abc\leq(2+a)(2+b)(2+c)$, followed by
$(u+v)^{1-\eps}\leq u^{1-\eps}+v^{1-\eps}$, therefore bounds
its $F_\eps$ factor by the product of the three individual factors,
with the same defining constant $C_{K,\eps}$.
Every prime dividing its square part divides at least one whole ideal.
The support products in the third assertion bound its $G_\eps$ factor.
Multiplication of the preceding sums proves the claim.
\end{proof}

\begin{proof}[Proof of Theorem~\ref{thm:one-line}]
Fix $c>0$, to be specified below, and use $z=c^{-1}\log X\log_2X$.
Take the function $r$ from Lemma~\ref{lem:euler-error}, and put
\[
 R_{\aideal}=\prod_{\substack{\Norm\p<z\\\p\notin S}}
       (1-r_{\p}\chi_{\aideal}(\p))^{-1}.
\]
Fix a nonzero smooth $W$, with $0\leq W\leq1$, supported in $(1,2)$.
Write $I_W=\int W(t)\,dt$ and
\[
 S_0=\sum_{\aideal\in\Om}W(\Norm\aideal/X)R_{\aideal}^2,
 \qquad
 S_1=\sum_{\aideal\in\Om}W(\Norm\aideal/X)
                  L(1,\chi_{\aideal})R_{\aideal}^2.
\]
We seek a lower bound for $S_1/S_0$ with an error tending to zero.

Choose $B>1$ large enough for Lemma~\ref{lem:short-log}, so that
$y=(\log X)^B>z$ for all sufficiently large $X$, and let
\[
 L_y(1,\chi)=\prod_{\Norm\p\leq y}
                 (1-\chi(\p)/\Norm\p)^{-1}.
\]
Let $S_{1,y}$ denote the sum defining $S_1$ with $L_y$ in place of $L$.
Mertens' estimate gives $L_y(1,\chi)\ll_{K,B}\log_2X$ uniformly.
Lemma~\ref{lem:short-log} consequently gives
\[
 \frac{S_1}{S_0}=\frac{S_{1,y}}{S_0}+o(1),
 \qquad L(1,\chi_{\aideal})\ll_{K,B}\log_2X.
\]
The latter bound will also suffice for the frequency assertion.
The primes in $S$ have zero ideal coefficient and contribute no Euler factor.

Expand the Euler products, separating the primes in $S$.
Proposition~\ref{prop:orthogonality} and
Lemma~\ref{lem:euler-error} give
\begin{align}
 S_0&=\kappa_{\Om}I_WX
       \sum_{\mideal\nideal=\square}r_{\mideal}r_{\nideal}
              h_K(\mideal\nideal)
       +O(X^{1/2+\eps+2/c+o(1)}),\label{eq:one-S0-expansion}\\
 S_{1,y}&=\kappa_{\Om}I_WX
       \sum_{\lideal\mideal\nideal=\square}
       \frac{a_{\lideal}r_{\mideal}r_{\nideal}}{\Norm\lideal}
          h_K(\lideal\mideal\nideal)
       +O(X^{1/2+\eps+2/c+o(1)}).\label{eq:one-S1-expansion}
\end{align}
All ideals in these sums are prime to $S$.
The main terms are nonnegative.  In the second sum, retain only ideals
$\lideal$ whose prime factors have norm below $z$.

For $q=\Norm\p<z$, $\p\notin S$, and $r=1-q/z$, the resulting local
factors are
\begin{align}
 D_{\p}&=1-h_K(\p)+h_K(\p)\frac{1+r^2}{(1-r^2)^2},\notag\\
 N_{\p}&=1-h_K(\p)+\frac{h_K(\p)}2
       \left(\frac1{(1-q^{-1})(1-r)^2}
                    +\frac1{(1+q^{-1})(1+r)^2}\right).
       \label{eq:one-local-numerator}
\end{align}
These formulas average the two unramified signs and the ramified value
$0$.  They count all square terms in the two expansions.

We first compare the second moment with the averaging error.
Since $h_K(\p)=q/(q+1)$,
\[
 D_{\p}=\frac{1+r^2}{(1-r^2)^2}
       \left(1-\frac{3r^2-r^4}{(q+1)(1+r^2)}\right).
\]
The logarithm of the product of the second factors is
$O_K(\log\log z)$.  Partial summation in the prime ideal theorem gives
\begin{equation}\label{eq:one-second-size}
 \prod_{\substack{\Norm\p<z\\\p\notin S}}
       \frac{1+r_{\p}^2}{(1-r_{\p}^2)^2}
 =X^{\lambda/c+o(1)},
 \qquad
 \lambda=\int_0^1\log\frac{2-2u+u^2}{u^2(2-u)^2}\,du
        =2-3\log2+\frac\pi2.
\end{equation}
The logarithmic singularity at $u=0$ is integrable.
To justify partial summation there, split at $\Norm\p=\delta z$.
The lower range contributes
$O_K(\delta(1+|\log\delta|)z/\log z+\log z)$. The upper range uses the
prime ideal theorem uniformly.  First let $z\to\infty$, then
$\delta\to0^+$.  The integral follows from
$\int_0^1\log(1+t^2)\,dt=\log2-2+\pi/2$ and elementary integration
of $\log u$ and $\log(2-u)$.

Thus the error in \eqref{eq:one-S0-expansion} is smaller by a fixed power
of $X$ if
\begin{equation}\label{eq:c-condition}
 c>c_0:=2(3\log2-\pi/2),\qquad
 0<\eps<\frac12-\frac{2-\lambda}{c}.
\end{equation}
In this range,
\begin{equation}\label{eq:one-S0-size}
 S_0=\kappa_{\Om}I_WX
          \prod_{\substack{\Norm\p<z\\\p\notin S}}D_{\p}(1+o(1))
     =X^{1+\lambda/c+o(1)}.
\end{equation}
The relative error tends to zero as a fixed negative power of $X$.
The same error bound applies to \eqref{eq:one-S1-expansion}.

It remains to compute the quotient accurately enough to retain its
constant term.  Cancelling the common factor
$h_K(\p)(1+r^2)/(1-r^2)^2$ gives the exact identity
\begin{equation}\label{eq:one-local-identity}
 (1-q^{-1})\frac{N_{\p}}{D_{\p}}
 =1-
 \frac{\displaystyle\frac{(1-r)^2}{(q+1)(1+r^2)}
               +\frac{(1-r^2)^2}{q^2(1+r^2)}}
      {\displaystyle1+\frac{(1-r^2)^2}{q(1+r^2)}}.
\end{equation}
Putting $u=q/z$, we obtain, uniformly for $2\leq q<z$,
\[
 (1-q^{-1})\frac{N_{\p}}{D_{\p}}
 =1-\frac1z\frac{u}{2-2u+u^2}+O(z^{-2}).
\]
Indeed, replacing $q+1$ by $q$ costs $O(u^2/q^2)=O(z^{-2})$,
and the second term in the numerator is also $O(z^{-2})$.
The denominator perturbation itself is $O(u/z)$, while the leading
numerator perturbation is $O(u/z)$. Hence its effect after division is
$O(u^2/z^2)=O(z^{-2})$. Summing logarithms and using the prime ideal theorem
now gives
\[
 \log\prod_{\substack{\Norm\p<z\\\p\notin S}}
          (1-(\Norm\p)^{-1})\frac{N_{\p}}{D_{\p}}
 =-\frac{c_3+o(1)}{\log z},\qquad
 c_3=\int_0^1\frac{u}{2-2u+u^2}\,du
     =\frac\pi4-\frac{\log2}{2}.
\]
For this use of partial summation the integrand is continuous on $[0,1]$.
The omitted fixed primes and the summed $O(z^{-2})$ errors contribute
$o(1/\log z)$.

For fixed $K$, Mertens' product and the classical zero-free region for
$\zeta_K$ give the stronger remainder
\[
 \prod_{\substack{\Norm\p<z\\\p\notin S}}
             (1-(\Norm\p)^{-1})^{-1}
 =e^\gamma\kappa_K\log z
       \prod_{\p\in S}(1-(\Norm\p)^{-1})
       \left(1+O_K(e^{-c_K\sqrt{\log z}})\right)
\]
for some $c_K>0$ (a possible exceptional real zero is harmless because
$K$ is fixed). For instance, this follows by partial summation from the
fixed-field prime ideal theorem
$\psi_K(x)=x+O_K(xe^{-c_K\sqrt{\log x}})$, after adjusting $c_K$.
In particular the relative error is $o(1/\log z)$, which is the precision
needed to retain the additive constant below. See \cite{GarciaLee} for
number field Mertens formulas.
Combining the last two displays, and
\eqref{eq:one-S0-expansion}--\eqref{eq:one-S1-expansion} yields
\begin{equation}\label{eq:one-ratio}
 \frac{S_1}{S_0}\geq\mathcal M_{K,\Om}
    \bigl(\log_2X+\log_3X-c_3-\log c+o(1)\bigr).
\end{equation}
Here $\log z=\log_2X+\log_3X-\log c$. To preserve the fixed power
saving required in \eqref{eq:c-condition}, first fix any $c>c_0$ and
let $X\to\infty$. The resulting lower bound holds for every such fixed
$c$. We may then let $c\downarrow c_0$ (equivalently, use a diagonal
sequence if one wants a single choice depending on $X$). Since
$C_2=c_3+\log c_0$, this proves the maximum assertion.

For the frequency assertion, first fix $0<\eta'<\eta$ and put
$c=c_0e^{\eta'}$.
The lower bound in \eqref{eq:one-ratio} exceeds the required threshold
by $\mathcal M_{K,\Om}(\eta-\eta'+o(1))$.
Let $\mathcal H_\eta\subseteq\Om(X)$ be the set where that threshold is exceeded.
Subtracting the complementary contribution to $S_1$, and using
$L(1,\chi)\ll\log_2X$, gives
\[
 \sum_{\aideal\in\mathcal H_\eta}
 W(\Norm\aideal/X)R_{\aideal}^2\geq S_0X^{-o(1)}.
\]
The pointwise bound
\[
 R_{\aideal}^2\leq
 \prod_{\substack{\Norm\p<z\\\p\notin S}}(z/\Norm\p)^2
 =X^{2/c+o(1)}
\]
and \eqref{eq:one-S0-size} consequently imply
\[
 \frac{\#\mathcal H_\eta}{\#\Om(X)}
 \geq X^{-(2-\lambda)/c+o(1)}
 =X^{-e^{-\eta'}/2+o(1)}.
\]
For each fixed $0<\eta'<\eta$ this bound holds with a fixed power
saving in \eqref{eq:c-condition}. Letting $\eta'\uparrow\eta$ after
$X\to\infty$, and diagonalizing over a sequence $\eta'_j\uparrow\eta$ if
a single $X$-dependent choice is desired, gives
$X^{-e^{-\eta}/2+o(1)}$. The gap in \eqref{eq:c-condition} remains
bounded away from zero because the limiting value $\eta$ is fixed and
positive.
\end{proof}
\section*{Acknowledgements}
The authors thank OpenAI's ChatGPT for assistance with the presentation of the manuscript. Z. Dong is supported by the National
Natural Science Foundation of China (Grant No. 1240011770).

\end{document}